\documentclass[a4paper]{amsart}
\usepackage{amsmath,amsthm,amssymb,latexsym,epic,bbm,comment,color}
\usepackage{graphicx,enumerate,stmaryrd}
\usepackage[all,2cell]{xy}
\xyoption{2cell}

\usepackage{tikz-cd}

\newcommand{\Z}{\mathbb Z}
\newcommand{\Zo}{\mathbb Z^{\geq 0}}
\newcommand{\R}{\mathbb R}
\newcommand{\K}{\mathbb K}
\newcommand{\C}{\mathbb C}
\newcommand{\al}{\alpha}

\newcommand{\pp}{\mathsf{p}}
\newcommand{\g}{\mathfrak{g}}
\newcommand{\p}{\mathfrak{p}}

\newcommand{\aaa}{\mathfrak{a}}

\newcommand{\mcA}{\mathcal A}

\newcommand{\mcD}{\mathcal D}
\newcommand{\mcF}{\mathcal F}
\newcommand{\mcG}{\mathcal G}
\newcommand{\mcJ}{\mathcal J}
\newcommand{\mcM}{\mathcal M}
\newcommand{\mcN}{\mathcal N}
\newcommand{\mcO}{\mathcal O}
\newcommand{\mcP}{\mathcal P}
\newcommand{\mcE}{\mathcal E}
\newcommand{\mcU}{\mathcal U}
\newcommand{\mcV}{\mathcal V}
\newcommand{\mcI}{\mathcal I}
\newcommand{\mcS}{\mathcal S}
\newcommand{\mcR}{\mathcal R}

\newcommand{\pa}{\partial}

\newcommand{\T}{T}
\newcommand{\F}{\mathrm{F}}

\newcommand{\gr}{\mathrm{gr}}
\newcommand{\red}{\mathrm{red}}
\newcommand{\dd}{\mathrm{d}} %de Rham differential

\newcommand{\pr}{\mathrm{pr}}

\newcommand{\id}{\mathrm{id}}
\newcommand{\pt}{\mathrm{pt}}

\newcommand{\supp}{\mathrm{supp}}

\newcommand{\ppc}{\pp} %letter for the covering projection

\newcommand{\rd}[1]{\left[{#1}\right]} 
\newcommand{\Lie}{\operatorname{Lie}}

\newcommand{\Hom}{\operatorname{Hom}}
\newcommand{\Der}{\operatorname{Der}}

\newcommand{\ad}{\operatorname{ad}}
\newtheorem{theorem}{Theorem}
\newtheorem{lemma}[theorem]{Lemma}
\newtheorem{definition}[theorem]{Definition}
\newtheorem{corollary}[theorem]{Corollary}
\newtheorem{proposition}[theorem]{Proposition}

\newtheorem{example}[theorem]{Example}

\newtheorem{remark}[theorem]{Remark}

\begin{document}
\title[Graded covering of a supermanifold]
{Graded covering of a supermanifold}

\dedicatory{Dedicated to the memory of Arkady Onishchik, 1933--2019}

\author{Elizaveta Vishnyakova}

\begin{abstract}
	
We introduce and investigate  the notion of a $\Zo$-graded covering for a supermanifold. More precisely,  Donagi and Witten \cite{Witten Atiyah classes} suggested a  construction of the first obstruction class for splitting of a  supermanifold via differential operators. We prove that an infinite prolongation of this construction satisfies some universal properties and can be seen as a covering of a supermanifold in the category of graded manifolds.

\end{abstract}

\maketitle

\section{Introduction}
Let $M$ be a connected smooth manifold. A {\bf universal covering} of $M$ consists of a connected, simply connected manifold $C$ and a local diffeomorphism $p:C\to M$ with some additional conditions.   A classical example of this situation is the following map: $p:\mathbb R\to S^1$, where $p(t) = e^{it}$. A universal covering $C$ of $M$ is {\bf (a)} unique up to isomorphism and {\bf (b)} any smooth map $f:N\to M$ of a simply connected, connected manifold $N$ to $M$ factors through $p$. Covering spaces are very important for algebraic topology, differential geometry, theory of Lie groups, Riemann surfaces, geometric group theory. More information about covering spaces and precise definitions can be found in \cite[Chapter 5]{Massey}.

There are examples of a covering spaces in algebra. (In algebra the word ``cover'' is preferred.) For instance if $F$ is a family of modules over a certain ring $R$, then an {\bf $F$-cover} of an $R$-module $M$ is a surjective homomorphism $p: C\to M$  with the following properties:
\begin{itemize}
	\item $C$ is in the family $F$;
	\item Any endomorphism $f$ of $C$ commuting with the map $p$, that is $p\circ f= p\circ \id$, is an automorphism. (This implies that $C$ is unique up to isomorphism. An analogue of property {\bf (a)} of a geometric covering.)
	\item Any (surjective) homomorphism from a module $N\in F$ to $M$ factors through $C$ (an analogue of property {\bf (b)} of a geometric covering).
\end{itemize} 

One of the most important examples of such covers are a torsion-free cover and a flat cover. For instance over an integral domain, every module $M$ has a torsion-free cover $C$, see \cite{Ban}. The flat cover conjecture due to Enochs \cite[Page 196, 1981]{Eno1} stats that 
for general rings, every module $M$ has a 
flat cover $C$. This conjecture was resolved positively by Bican, El Bashir and Enochs in \cite[2001]{Eno2}.

In the present paper we suggest a definition of  a {\bf $\Zo$-graded covering for any supermanifold}, which has properties similar to {\bf (a)} and {\bf (b)} of a topological and algebraic cover(ing). In more details, we show that 
\begin{enumerate}
	\item[{\bf (a)}] A $\Zo$-graded covering $\mcP$ is unique up to isomorphism for any supermanifold $\mcM$.
	\item[{\bf (b)}] For any graded manifold $\mcN$, we can factor the morphism $\psi:\mcN\to \mcM$ of $\Z_2$-graded ringed spaces through the covering projection $\pp:\mcP\to \mcM$. 
\end{enumerate}

Further we give two constructions of a $\Zo$-graded covering for any supermanifold. In other words, we prove that for any supermanifold $\mcM$, there exists a $\Zo$-graded covering $\mcP$. The ringed space $\mcP$ is an infinite dimensional graded manifold of the form $\mcP= \varprojlim \mcP_n$, where $\mcP_n$ is a graded manifold of degree $n$. We show that a construction given by Donagi and Witten  is equivalent to a construction of the graded manifold $\mcP_2$. The construction of $\mcP_2$ is related also to constructions of a metric double vector bundle and a skew-symmetric double vector bundle  obtained independently in \cite{JL,Fernando}. A construction of $\mcP_n$ for any $n$ can be deduced from results \cite{Grab,Vish} and was obtained in \cite{RVi}, however the meaning and universal properties of the inverse limit  $\mcP= \varprojlim \mcP_n$ were not understood before.

We investigate the case of a Lie supergroup. 
More precisely, for any Lie supergroup $\mcG$ we define a graded Lie group structure on its $\Zo$-graded covering $\mcP$. Further we show that the Lie superalgebra of  $\mcP$ is an example of a loop algebra constructed by V.~Kac, see also Allison, Berman, Faulkner, Pianzola \cite{Allison} and \cite{Eld} for applications of this construction. For instance the loop algebra of a (Lie) superalgebra was used by several authors to investigate graded-simple (Lie) algebras.

Our results are especially interesting in the complex-analytic (and algebraic) cate\-gory. The reason is the following. According to the Batchelor--Gawedzki Theorem any smooth supermanifold  is (non-canonically) split, that is its structure sheaf is isomorphic to the wedge power of a certain vector bundle. Therefore very often  we can study geometry of a split super\-manifold using geometry of vector bundles. This is not the case in the complex-analytic situation. The study of non-split super\-manifolds was initiated in \cite{Ber,Green}, where the first non-split supermanifold was described. Significant advances in this direction were achieved by A.L.~Onishchik. Interest in this problem arose again after Donagi and Witten's papers \cite{Witten not projected,Witten Atiyah classes}, where they proved that the moduli space of super Riemann surfaces for $g\geq 5$ is non-split (or more generally not projected) and studied other properties of non-split supermanifolds. Other examples of non-split supermanifolds are all (except of few exceptional cases) super-grassmannians and flag supermanifolds.

Our covering projection $\pp:\mcP\to \mcM$ induces an embedding of the corresponding structure sheaves $\pp^*:\mcO_{\mcM} \hookrightarrow \mcO_{\mcP}$. In other words the $\Zo$-covering $\mcP$ keeps the whole information about the (non-split) supermanifold $\mcM$.  Therefore we can study geometry of a non-split supermanifold in the category of graded manifolds using the tools of classical complex geometry due to the fact that geometrically a graded manifold is determined by a family of vector bundles.

\bigskip

\textbf{Acknowledgments:} E.V. was partially  supported by 
Coordenação de Aperfei\c{c}oa\-mento de Pessoal de N\'{\i}vel Superior - Brasil (CAPES) -- Finance Code
001, (Capes-Humboldt Research Fellowship).  
The author thanks Peter Littelmann for his suggestion to investigate the case of a Lie supergroup and a Lie superalgebra.   This helped us to understand that  the limit $\mcP= \varprojlim \mcP_n$ satisfies universal properties and to discover a $\Zo$-graded covering. We also thank Alexey Sharapov, Dimitry Leites and Mikhail Borovoi for useful comments.

\section{Preliminaries}

\subsection{Lie superalgebras and graded Lie superalgebras}
Throughout the paper we work over the field $\mathbb K=\mathbb C$ or $\mathbb R$. More information about supergeometry and Lie superalgebras can be found in \cite{Kac,BLMS,Bern,Leites,Var}. A {\it vector superspace} $V$ is a  $\mathbb Z_2$-graded vector space, this is $V = V_{\bar 0} \oplus V_{\bar 1}$, where $\bar i\in \Z_2=\{\bar 0,\bar 1\}$. Elements of $V_{\bar i}\setminus \{0\}$ are called homogeneous of {\it parity} $\bar i\in \Z_2$. To $v \in V_{\bar i}\setminus \{0\}$ we assign the parity $|v|:=\bar i$.
\begin{definition}
	A Lie superalgebra is a vector superspace 
	$\mathfrak g = \mathfrak g_{\bar 0}\oplus \mathfrak g_{\bar 1}$ with  a $\Z_2$-graded bilinear operation $[\,\,, \,]: \mathfrak g\times \mathfrak g\to \mathfrak g$, this is $[\mathfrak g_{\bar i}, \mathfrak g_{\bar j}]\subset \mathfrak g_{\bar i+\bar j}$, satisfying the following conditions
	\begin{align*}
	[x,y] &=-(-1)^{|x||y|}[y,x],\\
	[x, [y, z]] &= [[x, y], z] +  (-1)^{|x||y|} [y, [x, z]] ,
	\end{align*}
	where $x$, $y$, and $z$ are homogeneous elements. 
\end{definition}

We will consider Lie superalgebras with a parity-matched  $\Z$-grading, that is $\g = \bigoplus\limits_{n\in \Z} \g_n$, where $\g_{\bar 0} =  \bigoplus\limits_{n=2q}\g_q$ and  $\g_{\bar 1} =  \bigoplus\limits_{n=2q+1}\g_q$.

\begin{definition}
	Let $\g$ be an $\Z$-graded vector space (or a Lie superalgebra). Then we put
	$$
	\supp(\g)=\{n\in \Z\,\,|\,\, \g_{n}\ne \{0\} \}.
	$$
\end{definition}

In this paper most of the time we will work with non-negatively $\mathbb Z$-graded Lie superalgebras. That is $\Z$-graded Lie superalgebras with 
$$
\supp(\g) =\Zo:= \{0,1,\ldots, n,\ldots\}.
$$

\subsection{Supermanifolds, graded manifolds and graded supermanifolds}

\subsubsection{Supermanifolds}\label{sec supermanifolds def}

We consider a smooth or complex-analytic supermanifold in the sense of Berezin and Leites \cite{Bern,Leites}, see also \cite{BLMS}. Thus, a {\it superdomain in $\mathbb{C}^{n|m}$} is a ringed space $(U,\mathcal F\otimes \bigwedge (\mathbb{C}^{m})^*)$, where $U$ is an open set in $\mathbb{C}^{n}$ and $\mathcal F$ is the sheaf of smooth or holomorphic functions on $U$.   A supermanifold
$\mcM = (\mcM_0,\mathcal{O}_{\mcM})$ of dimension $n|m$ is a locally $\mathbb{Z}_2$-graded
ringed space that is locally isomorphic to a super\-domain in
$\mathbb{C}^{n|m}$. Here the underlying space $\mcM_0$ is a complex-analytic or smooth manifold.  The dimension $n$ of the underlying manifold $\mcM_0$ is called {\it even dimension of $\mcM$}, while $m$ is called {\it odd dimension of $\mcM$}. A
{\it morphism} $F:\mcM \to \mathcal{M}'$ of two
supermanifolds is a morphism between $\mathbb{Z}_2$-graded locally ringed spaces, this is, a pair $F = (F_{0},F^*)$, where $F_{0}:\mcM_0\to \mcM'_0$ is a holomorphic mapping and $F^*: \mathcal{O}_{\mcM'}\to (F_{0})_*(\mathcal{O}_{\mathcal M})$ is a homomorphism of sheaves of
$\mathbb{Z}_2$-graded locally ringed spaces. A morphism $F:\mcM \to \mcM'$ is called an {\it isomorphism} if there exists a morphism $G: \mcM' \to \mcM$ such that $G\circ F= \id_{\mcM}$ and $F\circ G= \id_{\mcM'}$. 
A supermanifold $\mcM$  is called {\it split}, if its structure sheaf is isomorphic to $\bigwedge \mathcal E$, where $\mathcal E$ is a locally free sheaf, this is a sheaf of sections of a  (holomorphic or smooth)  vector bundle $\mathbb E$.  In this case the structure sheaf is $\Z$-graded, not only $\Z_2$-graded.  According to the Batchelor--Gawedzki Theorem any smooth supermanifold  is  split. This is not true in the complex-analytic case, see \cite{Gaw, Ber,Green}.

Let again $\mathcal M=(\mathcal M_0,\mathcal O)$ be a supermanifold. (Sometimes we will omit $\mcM$ in $\mathcal O_{\mcM}$ if the meaning of $\mathcal O$ is clear from the context.) The {\it tangent sheaf} of a supermanifold $\mcM$ is the sheaf $\mathcal T := \mathcal{D}er\mcO$ of
derivations of the structure sheaf $\mcO$. Its global sections are called {\it smooth or holomorphic vector fields} on $\mcM$. Let $x\in \mcM_0$ and let $\mathfrak m_x$ be the maximal ideal in the local superalgebra $\mathcal O_x$. Then as in the classical geometry we can define the tangent space $T_x(\mcM)$ at the point $x$ as $T_x(\mcM):=(\mathfrak m_x/\mathfrak m^2_x)^*$. The tangent sheaf $\mathcal T$ is a locally free $\mcO$-sheaf. If $\mathcal U$ is a local chart on $\mcM$ with even and odd coordinates $(x_a,\xi_b)$, then $\mathcal T|_{\mathcal U}$ is a free $\mcO$-sheaf with a basis $(\frac{\partial}{\partial x_a}, \frac{\partial}{\partial \xi_b})$.

\subsubsection{The functor $\gr$}\label{sec gr}
 To any supermanifold we can assign a split supermanifold.  Let us briefly recall this construction. Let $\mathcal M=(\mcM_0,\mathcal O)$ be a supermanifold.  Then its structure sheaf has the following filtration
\begin{equation}\label{eq filtration}
\mathcal O = \mathcal J^0 \supset \mathcal J \supset \mathcal J^2 \supset\cdots \supset \mathcal J^p \supset\cdots,
\end{equation}
where $\mathcal J$ is the sheaf of ideals generated by odd elements in $\mathcal O$. We define
$$
\mathrm{gr} (\mathcal M): = (\mcM_0,\mathrm{gr}\mathcal O),\quad \text{where} \quad 
\mathrm{gr}\mathcal O = \bigoplus_{p \geq 0} \mathcal J^p/\mathcal J^{p+1}.
$$
The supermanifold $\mathrm{gr}(\mathcal M)$ is split,  that is its structure sheaf is isomorphic to $\bigwedge \mathcal E$, where $\mathcal E= \mathcal J/\mathcal J^{2}$ is a locally free sheaf. Since any morphism $F:\mcM\to\mcM'$ of supermanifolds preserves the  filtration (\ref{eq filtration}), the morphism $\gr (F):\gr\mcM\to\gr\mcM'$ is defined. Summing up, we defined the functor $\mathrm{gr}$  from the category of supermanifolds to itself.

\subsubsection{Graded manifolds and graded supermanifolds of degree $n$}\label{sec Graded manifolds of type Delta}
Let us start with a notion of a graded manifold  of degree $n$,
see \cite{Roytenberg} for more information.
 Consider the following $\Z$-graded vector superspace $V$ with $\supp (V)\subset \{0,1,\ldots, n\}$:
\begin{equation}\label{eq V}
V= \bigoplus_{i =0}^n V_{i}, \quad V=V_{\bar 0}\oplus V_{\bar 1},\quad V_{\bar 0}= \bigoplus_{i =2q} V_{i}, \quad V_{\bar 1}= \bigoplus_{i =2q+1} V_{i},
\end{equation}
where $V_i$ is a finite-dimensional vector space for any $i\in \Z$. To any element $v$ in  $V_i\setminus \{0\}$ we assign the weight $i\in \Z$ and the parity $|v| = \bar i\in \Z_2$.  Denote by $S^*(V)$ the super-symmetric algebra of $V$. In more details, $S^*(V)$ is the quotient of the tensor algebra of $V$ by the ideal generated by $x\otimes y- (-1)^{|x||y|}y\otimes x$, where $x,y$ are homogeneous elements. If $v = v_1 \cdots v_k\in S^*(V)$ is a product of homogeneous elements $v_i\in V_{q_i}\setminus \{0\}$, then we assign to $v$ the weight $\sum\limits_iq_i\in \Z$ and the parity $\sum\limits_i|q_i|\in \Z_2$. This induces a $\Z$-grading and $\Z_2$-grading in $S^*(V)$.

Let us fix a $\Z$-graded vector superspace $V$ over $\K$ with $\supp (V)\subset \{0,\ldots, n\}$. Consider a $\mathbb Z$-graded ringed space $\mathcal V = (\mathcal V_0,\mathcal O_{\mathcal V})$, where $\mathcal V_0 \subset V^*_0$ is an open set, 
$$
\mathcal O_{\mathcal V}: = \mathcal F_{\mathcal V_0}\otimes_{S^*(V_0)} S^*(V)
$$   
and
$\mathcal F_{\mathcal V_0}$ is the sheaf of smooth or holomorphic  functions on $\mathcal V_0$.  
We call the ringed space $\mathcal V$ a  {\it graded domain of degree $n$ and of dimension $\{n_{i}\}$}, where $n_{i} := \dim V_{i}$, $i=0,\ldots, n$. Further, let us choose a basis  $(x_a)$, $a=1,\ldots, n_0$, in $V_0$ and a basis $(\xi_{b_i}^{i})$, where $b_i=1,\ldots, n_i$, in $V_{i}$ for any $i=1,\ldots, n$. Then the set $(x_a, \xi_{b_i}^{i})$ is by definition a system of local coordinates in the graded domain $\mathcal V$. We assign the weight $0$ and the parity $\bar 0$ to $x_a$ and the weight $i$ and the parity $\bar i$ to $\xi_{b_i}^{i}$. Such coordinates as $(x_a, \xi_{b_i}^{i})$ in $\mathcal V$ are called {\it graded}.

Let $\mcV$ and $\mcV'$ be two graded domains with graded coordinates $(x_a, \xi_{b_i}^{i})$ and $(y_c, \eta_{d_j}^{j})$, respectively. A {\it morphism $\Phi: \mcV\to \mcV'$ of graded domains} is a morphism of the corresponding $\Z$-graded ringed spaces such that $\Phi^*|_{(\mcO_{\mcV'})_0}: (\mcO_{\mcV'})_0 \to (\Phi_0)_*(\mcO_{\mcV})_0$ is local, that is it is a usual morphism of smooth or holomorphic domains. Clearly such a morphism is determined by images of local coordinates $\Phi^*(y_c)$ and  $\Phi^*(\eta_{d_j}^{j})$. Conversely, if we have the following set of functions 
\begin{equation}\label{eq mor of graded domains}
\Phi^*(y_c)\in (\mcO_{\mcV})_0(\mcV_0)\quad  \text{and } \quad \Phi^*(\eta_{d_j}^{j})\in (\mcO_{\mcV})_j(\mcV_0),\quad j>0,
\end{equation}
such that $(\Phi^*(y_1)(u),\ldots,\Phi^*(y_{n_0})(u))\in \mcV'_0$ for any $u\in \mcV_0$, than there exists unique morphism $\Phi:  \mcV\to \mcV'$   of graded domains compatible with (\ref{eq mor of graded domains}). A {\it graded manifold of degree $n$ and of dimension $\{n_{i}\}$, $i=0,\ldots, n$,} is a $\mathbb Z$-graded ringed space $\mathcal N = (\mathcal N_0, \mathcal O_{\mathcal N})$, that is locally isomorphic to a graded domain of degree $n$ and of dimension $\{n_{i}\}$, $i=0,\ldots, n$.  More precisely, we can find an atlas $\{U_i\}$ of $\mcN_0$ and  isomorphisms $\Phi_i :(U_i, \mcO_{\mcN}|_{U_i}) \to \mcV_i$ of $\Z$-graded ringed spaces such that $\Phi_i\circ (\Phi_j)^{-1}: \mcV_j\to \mcV_i$ is a morphism of graded domains. A  {\it morphism of graded manifolds $\Phi=(\Phi_0,\Phi^*):\mcN\to \mcN_1$} is a morphism of the corresponding $\mathbb Z$-graded ringed spaces, which is locally a morphism of graded domains.

Since $\mcN$ is a $\mathbb Z$-graded ringed space, its structure sheaf is $\Z$-graded: $\mcO_{\mcN} = \bigoplus\limits_{q\in \Z} (\mcO_{\mcN})_q$.
Denote by $\mcI_n$ the sheaf of ideals in $\mcO_{\mcN}$ generated by $(\mcO_{\mcN})_q$, where $q\geq n+1$. The $\Z$-graded ringed spaces $\mcN^{(n)}:=(\mcN_0,\mcO_{\mcN}/\mcI_n)$   are locally ringed spaces for any $n\geq 0$. If we have a morphism  $\Phi:\mcN\to \mcN_1$ of graded manifolds, than the induced morphism $\mcN^{(n)}\to \mcN_1^{(n)}$ is local for any $n$.

\begin{remark}\label{rem projection of graded}
	 To any graded manifold of degree $n$ we can assign a graded manifold of degree $0\leq n'<n$. In this case we have a natural morphism $\pr^n_{n'}:\mathcal N\to \mathcal N'$, which is called projection.  A detailed description of this construction can be found for example in \cite[Section 4.1]{Vish}.
\end{remark}

A definition of a graded supermanifold is the following. 
Consider a $\Z$-graded vector superspace with support $\{0,\ldots, n\}$:
$$
V= V_0\oplus V_1\oplus \cdots \oplus V_n,
$$
where $V_i= (V_i)_{\bar 0}\oplus (V_i)_{\bar 1}$ are a superspace. Similarly to above we define a graded superdomain of degree $n$ and a graded supermanifold of degree $n$.    

 \subsection{$\Z_2$-graded morphisms} 
 
 Let $\mcM=(\mcM_0,\mcO_{\mcM})$ be a supermanifold and $\mcN=(\mcN_0,\mcO_{\mcN})$ be a graded manifold (of degree $n$).
 Let us define a $\Z_2$-graded morphism $\psi:\mcN\to \mcM$. To understand the nature of this problem, consider the following example. 
 
 \begin{example}\label{ex sin(y^0+y^2)}
 	Let $\mcM= \K^{1|0}$, $\K=\R$ or $\C$, and $\mcN$ be a graded domain $V_0\oplus V_2$, where $\dim V_i=1$. Denote by $x$ the standard coordinate in $\mcM$ and by  $y^0$, $y^2$  graded coordinates in $\mcN$ with weights $0$ and $2$, respectively. Let us try to define a morphism $\psi: \mcN\to \mcM$ as a morphism of $\Z_2$-graded ringed spaces. (Clearly $\mcN$ is naturally $\Z_2$-graded.) Let we have
 	$$
 	\psi^*(x) = y^0+y^2.
 	$$ 
 	Then it has to be
 	$$
 	\psi^*(\sin x) = \sin (y^0+y^2) = y^0+y^2 - \frac{(y^0+y^2)^3}{3!} + \cdots.
 	$$
 	We see that the image of $\sin x$ has to be an infinite series in $y^2$. However such an infinite series if not an element in the structure sheaf $\mcO_{\mcN}$.  
 \end{example}

 Example \ref{ex sin(y^0+y^2)} shows that if the image $\psi^*(F)$, where $F\in \mcO_{\mcM}$, is determined by  images of local coordinates  we have to consider infinite series in graded coordinates or we need to modify the definition of a morphism $\psi$.

 Let as above $\mathcal I_k$ be the sheaf of ideals in $\mcO_{\mcN}$ generated by all $(\mcO_{\mcN})_q$, where $q\geq k+1$. The ringed spaces $\mcN^{(k)}:= (\mcN_0, {\mcO}_{\mcN}/{\mathcal I}_k)$ are $\Z$-graded (therefore $\Z_2$-graded) and local for any $k\geq 0$.  Note that we have the following natural isomorphism of sheaves of vector spaces
 \begin{equation}\label{eq O/I = O_0+O_1+ ... +O_n}
 {\mcO}_{\mcN}/{\mathcal I}_k \simeq ({\mcO}_{\mcN})_0\oplus ({\mcO}_{\mcN})_1\oplus \cdots \oplus  ({\mcO}_{\mcN})_k.
 \end{equation}
 Within the framework of this paper, it is natural to consider the following definition.

 \begin{definition}\label{def morphism N to M}
 	A morphism of a graded manifold $\mcN$ to a supermanifold $\mcM$ is a compatible family of morphisms  $(\psi^{(k)}: \mcN^{(k)}\to \mcM)$, $k \geq 0$, of $\Z_2$-graded locally ringed spaces.  
 \end{definition}
 
 Let us comment the compatibility of morphisms $\psi^{(k)}$ in Definition \ref{def morphism N to M}. The family of morphisms $(\psi^{(k)}: \mcN^{(k)}\to \mcM)$ is called {\it compatible} if  the following diagrams are commutative for any $k\geq 0$:
 $$
 \begin{tikzcd}
 & \mcO_{\mcN}/\mcI_{k+1}  \arrow[d, "\pr^{k+1}_k"]\\
 \mcO_{\mcM}  \arrow[ur,  "(\psi^{(k+1)})^*"]   \arrow[r, "(\psi^{(k)})^*"'  ]  & \mcO_{\mcN}/\mcI_{k} ,
 \end{tikzcd}
 $$
 where $\pr^{k+1}_k: \mcO_{\mcN}/\mcI_{k+1}  \to \mcO_{\mcN}/\mcI_{k}$ is the natural projection. Note that in this case the base morphism $\psi_0: \mcN_0\to\mcM_0$ is the same for any $k$.

 \begin{theorem}\label{theor Leites analog psi is unique} Let $\mcM$ be a supermanifold, $\mcN$ be a graded manifold and  $(\psi^{(k)}: \mcN^{(k)}\to \mcM)$ be a morphism as in Definition \ref{def morphism N to M}.

 {\bf (1)}	Any morphism  $\psi^{(k)}: \mcN^{(k)}\to \mcM$, $k \geq 0$, of $\Z_2$-graded locally ringed spaces is locally determined by the images of local coordinates. 
 	
 {\bf (2)}	The compatible family $(\psi^{(k)}: \mcN^{(k)}\to \mcM)$,  is locally determined by the images of local coordinates.
 \end{theorem}
 
 \begin{proof} Clearly the second statement is the consequence of the first one. Let u sprove the first statement. 
 	The idea of the proof is similar to \cite[Section 2.1.7, Proof of Theorem]{Leites}. It is sufficient to assume that $\mcM=\mcU=(\mcU_0,\mcO_{\mcU})$ is a  superdomain with coordinates $(x_a,\xi_b)$. Let us prove that $(\psi^{(k)})^*(x_a)$ and $(\psi^{(k)})^*(\xi_b)$ determines $\psi^{(k)}=(\psi_0,(\psi^{(k)}))^*$ for any $k\geq 0$.

 	Let us choose $u\in \mcU_0$ and  $v\in \mcN_0$ such that $\psi_0(v)=u$. Since $\psi^{(k)}$ is local, we have 
 	$
 	(\psi^{(k)})^*(\mathfrak m_u)\subset \mathfrak m_v,
 	$
 	where $\mathfrak m_o$ is the maximal ideal in the corresponding structure sheaf at the point $o$. 
 	Further, by Hadamard's lemma for supermanifolds for any $F\in \mcO_{\mcM}$ we can find a polynomial $P$ of degree $q$ such that $F-P\in \mathfrak m_u^q$.
 	Therefore, $(\psi^{(k)})^* (F)- (\psi^{(k)})^* (P) \in \mathfrak m_v^q$. Now using (\ref{eq O/I = O_0+O_1+ ... +O_n}) we identify $(\psi^{(k)})^* (F)$ with an element 
 	$$
 	H=H_0+\cdots +H_k\in (\mcO_{\mcN})_0\oplus \cdots \oplus (\mcO_{\mcN})_k, \quad H_i\in (\mcO_{\mcN})_i.
 	$$
 	Denote by $(z^{s}_{i_s},\zeta^t_{j_t})$, where $s$ is even and $t$ is odd, local graded coordinates with weights $s$ and $t$, respectively, in a neighborhood of $v$. Then in this neighborhood $H$ can be regarded as a superfunction in even and odd coordinates $(z^{s}_{i_s},\zeta^t_{j_t})$. Let we have another morphism  $(\tilde\psi^{(k)})^*\ne (\psi^{(k)})^*$ such that 
 	$$
 	(\tilde\psi^{(k)})^*(x_a) = (\psi^{(k)})^*(x_a) ,\quad (\tilde\psi^{(k)})^*(\xi_b) =  (\psi^{(k)})^*(\xi_b).
 	$$ 
 	Then we have
 	$$
 	(\tilde\psi^{(k)})^* (F)\mod \mathfrak m_v^q\,\, \equiv \,\, (\psi^{(k)})^* (P) \mod \mathfrak m_v^q\,\, \equiv \,\, (\psi^{(k)})^* (F)\mod \mathfrak m_v^q
 	$$
 	for any $v$ and any $q$. The result follows from \cite[Section 2.1.7]{Leites}.
 \end{proof}

 \begin{theorem}\label{theor Leites analog psi is existence}
 	Let $(x_a,\xi_b)$ be coordinates in a superdomain $\mcU$ and $\mcV=(\mcV_0,\mcO_{\mcV})$ be a graded domain. Consider a set of functions satisfying the following conditions 
 	\begin{equation}\label{eq x^* xi^*}
 	\begin{split}
 	& (x^*_a)^{(k)} \in  ((\mcO_{\mcV})/\mcI_k)_{\bar 0}, \quad \pr_k^{k+1}((x^*_a)^{(k+1)}) = (x^*_a)^{(k)}, \quad k\geq 0; \\ 
 	&(\xi^*_b)^{(k)} \in ((\mcO_{\mcV})/\mcI_k)_{\bar 1}, \quad \pr_k^{k+1}((\xi^*_b)^{(k+1)}) = (\xi^*_b)^{(k)}, \quad k\geq 0,\\
 	& (x^*_a)^{(k)} (v)\in \mcU_0, \quad \text{for any} \quad v\in \mcV_0,
 	\end{split}
 	\end{equation}
 	where $\pr^{k+1}_k: \mcO_{\mcV}/\mcI_{k+1}  \to \mcO_{\mcV}/\mcI_{k}$ is the natural projection.
 	Then there exists unique compatible family of morphisms $\psi^{(k)}: \mcV^{(k)}\to \mcM$ such that 
 	\begin{equation}\label{eq psi(x)=x^*}
 	(\psi^{(k)})^* (x_a) =  (x^*_a)^{(k)}\quad  \text{and}  \quad (\psi^{(k)})^* (\xi_b) =  (\xi^*_b)^{(k)}. 
 	\end{equation}
 \end{theorem}

 \begin{proof}
 	If there exists such a family of morphisms $(\psi^{(k)})_{k\geq 0}$, it is compatible. Indeed, by (\ref{eq x^* xi^*}) and (\ref{eq psi(x)=x^*}) we have 
 	$$
 	(\psi^{(k)})^*(x_a)=\pr^{k+1}_k \circ(\psi^{(k+1)})^*(x_a),\quad  (\psi^{(k)})^*(\xi_b)=\pr^{k+1}_k \circ(\psi^{(k+1)})^*(\xi_b).
 	$$
 	Therefore, $\pr^{k+1}_k \circ(\psi^{(k+1)})^*= (\psi^{(k)})^*$ since  by Theorem \ref{theor Leites analog psi is unique} a morphism $\mcV^{(k)}\to \mcM$  satisfying (\ref{eq psi(x)=x^*}) is unique.

 	Now we use an argument from \cite[Section 2.1.11, Proof of Theorem]{Leites}. Let us repeat this argument here for completeness.  If $F(x_a,\xi_b)\in \mcO_{\mcU}$ we put 
 	$$
 	(\psi^{(k)})^*(F):= F((\psi^{(k)})^* (x_a), (\psi^{(k)})^*(\xi_b)) \mod \mcI_k,
 	$$
 	and $F((\psi^{(k)})^* (x_a), (\psi^{(k)})^*(\xi_b)) \mod \mcI_k$ we define by Taylor formula. This means that we consider the Taylor series of the function $F(z^0_a+ z'_a, \zeta_b)$, where $z^0_a$ is the part of weight $0$ and $z'_a$ is the rest, and we decompose $F$ meaning that $z'_a$ is a  small variation. Then we replace $z^0_a$ by $\pr_0\circ(\psi^{(k)})^* (x_a)$ and $z'_a $ by $(\psi^{(k)})^* (x_a) - \pr_0\circ(\psi^{(k)})^* (x_a)$, where $\pr_0:\mcO_{\mcV} \to (\mcO_{\mcV})_0$ the natural projection.  Since we consider the result modulo $\mcI_k$ the variation $z'_a$ is nilpotent and the series is well-defined. 
 \end{proof}

\subsection{Lie supergroups, graded Lie groups and super Harish-Chandra pairs}
\subsubsection{Lie supergroups}

A {\it Lie supergroup} is a group object in
the category of supermanifolds.
In other words a Lie supergroup is a supermanifold $\mathcal G=(\mathcal G_0, \mathcal O_{\mathcal G})$ with three morphisms:
$\mu:\mathcal G\times \mathcal G\to \mathcal G$
(the multiplication morphism), $\kappa:\mathcal G\to \mathcal G$ (the inversion morphism),
$\varepsilon:(\mathrm{pt},\K)\to \mathcal G$ (the identity
morphism). Moreover, these morphisms have to satisfy the usual
conditions, modeling the group axioms. The underlying manifold $\mathcal G_0$ of $\mathcal G$
is a smooth or complex-analytic Lie group. As in the theory of Lie groups and Lie algebras, we can define the Lie superalgebra $\mathfrak{g}=\Lie(\mcG)$ of a Lie supergroup $\mathcal G$. Let us remind this construction.

Let $\mcG$ be a Lie supergroup, $e=\varepsilon_0(\pt)$ be the identity element of $\mcG_0$ and $\mathfrak m_e$ be the maximal ideal in the local algebra $(\mathcal O_{\mathcal G})_e$. Then $T_e(\mcG):=(\mathfrak m_e/\mathfrak m^2_e)^*$ is the tangent space of $\mcG$ at $e$. By definition, $\mathfrak{g}:=\Lie(\mcG)$ is the subsuperalgebra of the Lie superalgebra of smooth or holomorphic vector fields on
$\mathcal G$ consisting of all right invariant
vector fields on $\mathcal G$. Recall that a vector field $Y$ on $\mcG$ is called {\it right invariant} if $(Y \times \id) \circ \mu^* = \mu^* \circ Y$. Any right invariant vector field $Y\in \mathfrak{g}$ has the following form
$$
Y=(v_e\times \mathrm{id})\circ \mu^*
$$
for a $v_e\in (\mathfrak m_e/\mathfrak m^2_e)^*$, see \cite[Theorem 6.1.1]{Var}.  Note that as in the classical case the vector superspace $(\mathfrak m_e/\mathfrak m^2_e)^*$ may be identified with the vector superspace of all maps $D: (\mathcal O_{\mathcal G})_e \to \K$ satisfying the super Leibniz rule. 
Further the correspondence $v_e \mapsto (v_e\otimes \mathrm{id})\circ \mu^*$ defines an isomorphism of $T_e(\mcG)$ onto $\g$. Similarly we can define a {\it graded Lie group of degree $n$} as a group object in the category of graded manifolds of degree $n$.

\subsubsection{Super Harish-Chandra pairs}
To study a Lie supergroup one uses the theory of super Harish-Chandra pairs, see \cite{Bern} and also \cite{Fioresi,ViLieSupergroup}.

\begin{definition}
	A  super Harish-Chandra pair $(G, \g)$  consists of 
	a Lie superalgebra $\g=\mathfrak{g}_{\bar
		0}\oplus\mathfrak{g}_{\bar 1}$,    a Lie group $G$    such that $\mathfrak{g}_{\bar 0} = \mathrm{Lie} (G)$, and
	an action $\alpha$  of $G$ on $\g$ by authomorphisms such that 
	\begin{itemize}
		\item for any $g\in G$ the action $\alpha(g)$ restricted to $\g_{\bar{0}}$ coincides with the adjoint action $Ad(g)$;
		\item the differential $(d\alpha)_e: \g_0 \to  \Der(\g)$ at the identity $e\in G$ coincides with
		the adjoint representation $\ad$ of $\mathfrak g_{\bar 0}$ in $\mathfrak g$.
	\end{itemize}
\end{definition}
Super Harish-Chandra pairs form a category.  
The following theorem was proved in \cite{ViLieSupergroup} for complex-analytic Lie supergroups.

\begin{theorem}\label{theor HC pairs}
	The category of complex-analytic Lie supergroups is
	equivalent to the category of complex-analytic super Harish-Chandra pairs.
\end{theorem}

A similar result in the real case was obtained in \cite{Kostant}, the algebraic case was treated by several authors, see for example \cite{Gav,Masuoka,MasuokaShibata} and references therein.

 A super Harish-Chandra pair of degree $n$ is a $\mathbb Z$-graded super Harish-Chandra pair $(G,\g)$, where $\g$ is a  $\Z$-graded Lie superalgebra with $\supp (\g) \subset \{0,\ldots, n\}$,  and $G=\Lie (\g_0)$. Moreover, $\alpha(g): \g \to \g$, where $g\in G$, is an automorphism of $\g$ preserving the $\Z$-grading and  $\dd_e\alpha (X) = [X, \cdot]_\g$ for any $X\in \g_0$.  In \cite[Theorem 5.6]{Jubin} it was noticed that an analogue of Theorem \ref{theor HC pairs} holds true for graded Lie groups of degree $n$.

\subsubsection{Construction of the Lie supergroup corresponding to a super Harish-Chandra pair}
Let us remind a reader how to construct a Lie supergroup $\mathcal G$ (or a graded Lie group of degree $n$) using a given super (or of degree $n$) Harish-Chandra pair $(G,\mathfrak g)$. Denote by $\mathcal U(\mathfrak h)$ the universal enveloping algebra of Lie superalgebra $\mathfrak h$.  We need to define a structure sheaf $\mathcal O$ of $\mathcal G$. In the super and graded cases we, respectively, put
\begin{equation}\label{eq structure sheaf of G}
\begin{split}
\mathcal O = \mathrm{Hom}_{\mathcal U(\mathfrak g_{\bar 0})} (\mathcal U(\mathfrak g), \mathcal F_{G}), \quad \mathcal O = \mathrm{Hom}_{\mathcal U(\mathfrak g_{0})} (\mathcal U(\mathfrak g), \mathcal F_{G}).
\end{split}
\end{equation}
Here $\mathcal F_G$ is the sheaf of (smooth or holomorphic) functions on $G$. Recall that in the super- and graded cases we respectively have $G=\Lie(\g_{\bar 0})$ and  $G=\Lie(\g_0)$. 
Using the Hopf superalgebra structure on $\mathcal U(\mathfrak g)$ we can define explicitly the multiplication morphism $\mu$, the inversion morphism $\kappa$ and the identity $\varepsilon$, see for instance \cite{ViLieSupergroup}. Indeed, assume that a super  or graded of degree $n$ Harish-Chandra pair $(G,\mathfrak g)$ is given. Let us define the supergroup structure of the corresponding Lie supergroup or graded Lie group  $\mathcal G$ of degree $n$. Let $X\cdot Y\in
\mathcal U(\mathfrak{g}\oplus \mathfrak{g})\simeq
\mathcal U(\mathfrak{g})\otimes \mathcal U(\mathfrak{g})$, where
$X$ is from the first copy of $\mathcal U(\mathfrak{g})$ and $Y$
from the second one, $f\in \mathcal{O}$, see (\ref{eq structure sheaf of G}), and $g,\,h\in
G$. The following formulas define a multiplication morphism, an
inverse morphism and an identity morphism respectively:
\begin{equation}\label{eq umnozh in supergroup}
\begin{split}
\mu^*(f)(X\cdot Y)(g,h)&=f(\alpha(h^{-1})(X)\cdot Y)(gh);\\
\kappa^*(f)(X)(g^{-1})&=f(\alpha(g^{-1})(S(X)))(g);\\
\varepsilon^*(f)&=f(1)(e).
\end{split}
\end{equation}
Here $S$ is the antipode map in $\mathcal U(\mathfrak{g})$ considered as a Hopf superalgebra and $\alpha$ is as in the definition of a super Harish-Chandra pair. Theorem \ref{theor HC pairs} or \cite[Theorem 5.6]{Jubin} imply that any Lie supergroup or graded Lie group is isomorphic to a Lie supergroup with the structure sheaf (\ref{eq structure sheaf of G}) and the structure morphisms (\ref{eq umnozh in supergroup}).

\subsection{Inverse limit of graded manifolds and graded Lie groups}\label{sec inverse limit} 

In this paper we will work with a special type of infinite dimensional graded manifolds or graded Lie groups, more precisely with the inverse limit of graded manifolds or of graded Lie groups with fixed base space $N$.  Related construction of a direct limit of (super)manifolds may be found for instance in \cite{DonPen,OniVi}.

Let we have a family of graded manifolds $\mcN_k=(N,\mcO_{\mcN_k})$ of degree $k=0,1, \ldots$. For any $k$ we have a projection $\pr^k_{k-1}$ of $\mcN_k$ to a graded manifold $\mcN'_{k-1}$ of degree $k-1$, see Remark \ref{rem projection of graded}. Assume that our family of graded manifolds $(\mcN_k)_{k\in \Z}$ is compatible in the following sense. For any $k\in \Zo$  we have $\pr^k_{k-1}(\mcN_k) = \mcN_{k-1}$. 
Therefore for any $n\in \Zo$ we have the following sequence of projections
$$
\mcN_0 \xleftarrow{\pr^1_0} \mcN_1 \xleftarrow{\pr^2_1} \mcN_2 \xleftarrow{\pr^3_2} \ldots \xleftarrow{\pr^n_{n-1}} \mcN_n.
$$
For the corresponding structure sheaves we have the following sequence of inclusions
\begin{equation}\label{eq direct sequence of sheaves}
\mcO_{\mcN_0} \xrightarrow{(\pr^1_0)^*} \mcO_{\mcN_1} \xrightarrow{(\pr^2_1)^*} \mcO_{\mcN_2} \xrightarrow{(\pr^3_2)^*} \cdots \xrightarrow{(\pr^n_{n-1})^*} \mcO_{\mcN_n}, \quad n\in \Zo. 
\end{equation}
We put
$$
\mcN: = (N, \mathcal O_{\mcN}), \quad \text{ where } \mcO_{\mcN} : = \bigcup_{k=0}^\infty \mcO_{\mcN_{k}}, 
$$
That is $f\in \mcO_{\mcN}$ if and only if there exist $k$ such that $f\in \mcO_{\mcN_k}$. In other words, $\mathcal O_{\mcN}$ is the {\it direct limit} of the {\it direct system} of sheaves (\ref{eq direct sequence of sheaves}). We will write
$$
\mcN = \varprojlim \mcN_k, \quad \mathcal O_{\mcN} = \varinjlim \mathcal O_{\mcN_k}. 
$$

A {\it morphism}  $\Phi: \mcN \to \mcN'$, where $\mcN= \varprojlim \mcN_k$ and $\mcN'= \varprojlim \mcN'_k$, is a family of morphisms $(\Phi_k: \mcN_k \to \mcN'_{k})$ 
such that 
$$
\Phi_{k}\circ \pr^{k+1}_{k} = \pr^{k+1}_{k} \circ  \Phi_{k+1},\quad k\in \Z.
$$

The inverse limit of graded Lie groups has a graded Lie group structure. Here we additionally assume that the projections $\pr^{k+1}_k$ are graded Lie group homomorphisms.

\section{A $\Z^{\geq 0}$-covering of a supermanifold}

In this section we introduce a new concept: {\bf a $\Z^{\geq 0}$-covering of a super\-manifold}. In Introduction we gave several examples of coverings in geometry and algebra. Consider for instance a flat cover $C$ of a module $M$. In this case we construct a cover of a module $M$ in the category of flat modules. This paper is devoted to the following question: can a supermanifold  be covered in the category of  graded manifolds? Such a covering if exists we will call a {\bf $\Z^{\geq 0}$-covering of a supermanifold}.

Our idea is the following. For any supermanifold $\mcM$ we define a $\Z^{\geq 0}$-covering of $\mcM$ as a certain non-negatively $\Z$-graded manifold $\mcP$ together with a covering projection $\pp:\mcP\to \mcM $. It has to satisfy the universal property similar to the geometric and algebraic cases. For instance any morphism $\psi: \mcN\to \mcM$ of a $\Z$-graded manifold $\mcN$ to the supermanifold $\mcM$ has to factor through $\pp$. In other words it has to exist a $\Z$-graded morphism $\Psi: \mcN\to \mcP$ such that $\pp\circ \Psi  =\psi$. 

Assume that we constructed a $\Z^{\geq 0}$-covering $\pp:\mcP\to \mcM $. 
Consider a graded manifold $\mcN=(\pt,\bigwedge (\zeta))$, where we assume that $\zeta$ is an odd variable of weight $2n+1$. Let $\psi:\mcN\to \mcM$ be a $\Z_2$-graded morphism of locally ringed spaces and $\Psi$ be a lift of $\psi$ such that $\pp\circ \Psi  =\psi$.  Since $\Psi$ has to be $\Z$-graded, this implies that $\mcP$ has local coordinates of weight $2n+1$. Similarly for even coordinates. This implies that $\mcP$ has to be infinite dimensional, since it has local coordinates of any weight in $\Zo$.

\subsection{Definition of a $\Z^{\geq 0}$-covering of a supermanifold} We suggest the following definition of a $\Z^{\geq 0}$-covering $\mcP$ of a supermanifold $\mcM=(\mcM_0,\mcO)$.  As we saw above the $\Z^{\geq 0}$-covering $\mcP$ of $\mcM$ cannot be finite dimensional.

  Let $\mcP = \varprojlim \mcP_n$ be the inverse limit of  graded manifolds $\mcP_n$ of degree $n$  with the same underlying space $(\mcP_n)_0=  \mcP_0$, see Section \ref{sec inverse limit}. In this case the structure sheaves $\mcO_{\mcP_n}$ of $\mcP_n$ and $\mcO_{\mcP}$ of $\mcP$ are $\Z$-graded. In other words, 
$$
\mcO_{\mcP_n}= \bigoplus\limits_{q\in \Z} (\mcO_{\mcP_n})_q,\quad \mcO_{\mcP}= \bigoplus\limits_{q\in \Z} (\mcO_{\mcP})_q .
$$ 
Let again $\mathcal I_k$ be the sheaf of ideals in $\mcO_{\mcP}$ generated by all $(\mcO_{\mcP})_q$, where $q\geq k+1$. We have the following natural isomorphism of sheaves of superalgebras 
\begin{equation}\label{eq iso of sheaves}
{\mcO}_{\mcP}/{\mathcal I}_k \simeq {\mcO}_{\mcP_n}/({\mathcal I}_k\cap {\mcO}_{\mcP_n}),\quad n\geq k\geq 0.
\end{equation}
Denote 
$$
\mcP^{(k)}: = (\mcP_0,{\mcO}_{\mcP}/{\mathcal I}_k )\simeq (\mcP_0, {\mcO}_{\mcP_n}/({\mathcal I}_k\cap {\mcO}_{\mcP_n})),\quad n\geq k. 
$$ 
Note that we have the following natural isomorphism of sheaves of vector spaces
\begin{align*}
&{\mcO}_{\mcP}/{\mathcal I}_k \simeq ({\mcO}_{\mcP})_0\oplus ({\mcO}_{\mcP})_1\oplus \cdots \oplus  ({\mcO}_{\mcP})_k, \quad k\geq 0.
\end{align*}

 An  atlas on $\mcP= \varprojlim \mcP_n$ is a set $\{\mathcal V_{\lambda} \}$ of graded domains $\mathcal V_{\lambda}$, where each $ \mathcal V_{\lambda}$ is the inverse limit of graded domains $(\mathcal  V_{\lambda})_n$ of $\mcP_n$ and $\{(\mathcal V_{\lambda})_n \}$ is an atlas of $\mcP_n$ for any $n$. Let $(y_{a_s}^s,\eta^t_{b_t})$, where $s$ is an even non-negative integer and  $t$ is an odd non-negative integer, be graded coordinates in $\mathcal V_{\lambda} = \varprojlim (\mathcal V_{\lambda})_n$. As above the weight of $y_{a_s}^s$ is equal to $s$ and the weight of $\eta^t_{b_t}$ is equal to $t$.  We call such coordinates {\it adapted}, if  $(y_{a_s}^s,\eta^t_{b_t})_{s,t\leq n}$ are graded coordinates in $(\mathcal V_{\lambda})_n$ for any $n$. 
  We will denote images of  $y_{a_s}^s$ or $\eta^t_{b_t}$, $s,t\leq k$, in the structure sheaf of $\mcP^{(k)}$ also by $y_{a_s}^s$ or $\eta^t_{b_t}$, respectively. For example for the image of $y_{a_s}^s$, $s\leq k$, in $\mcO_{\mcP^{(k)}}$ we use the same notation $y_{a_s}^s$.

\begin{definition}\label{def new-covering of supermanifold}  A $\Z^{\geq 0}$-covering  of a supermanifold $\mcM$ is an infinite dimensional graded manifold $\mcP= \varprojlim \mcP_n$ such that $\mcM_0=\mcP_0$ together with a compatible family of morphisms $\{\pp^{(k)} : \mcP^{(k)}\to \mcM\}_{k\in \Z^{\geq 0}}$ of $\Z_2$-graded ringed spaces such that
  we can choose atlases $\{\mathcal U_{\lambda}\}$ and $\{\mathcal V_{\lambda} = \varprojlim (\mathcal V_{\lambda})_n \}$ on $\mcM$ and $\mcP$, respectively, with $(\mathcal U_{\lambda})_0= (\mathcal V_{\lambda})_0$,  with even and odd coordinates $(x_a,\xi_b)$ of $\mathcal U_{\lambda}$ and  with adapted graded coordinates $(y_a^s,\eta^t_b)$ of $\mathcal V_{\lambda}$, where $s$ is an even non-negative integer and  $t$ is an odd non-negative integer,  such that
	$$
	\pr_s\circ (\pp^{(k)})^*(x_a) =  y_a^s, \quad \pr_t\circ (\pp^{(k)})^*(\xi_b) =  \eta^t_b,\quad s,t \leq k,
	$$
	where $\pr_q: \mathcal O_{\mathcal P^{(k)}}\to (\mathcal O_{\mathcal P^{(k)}})_q$, $q\in \Z$, is the natural projection. 
\end{definition}

\begin{remark}
\label{rem series for p^*}
	We can write locally the morphism $\pp^{(k)} : \mcP^{(k)}\to \mcM$, where $k\in \Z^{\geq 0}$, from 	Definition \ref{def new-covering of supermanifold} without use of projections $\pr_q$. More precisely, we can write $\pp^{(k)} : \mcV_{\lambda}^{(k)}\to \mcU_{\lambda}$ in  coordinates from Definition \ref{def new-covering of supermanifold}  in the following form
	\begin{equation}\label{eq pp^* without projections pr}
	\begin{split}
&	(\pp^{(k)})^*(x_a) =  y_a^0+ y_a^2+ y_a^4+\cdots + y_a^{2v} \mod \mcI_k; \\  
&	(\pp^{(k)})^*(\xi_b) =  \eta^1_b + \eta^3_b + \eta^5_b + \cdots  + \eta_b^{2v+1} \mod \mcI_k,
	\end{split}
	\end{equation}
where $v>k$ is any integer. 
Sometimes instead of (\ref{eq pp^* without projections pr}) we will use the following schematic series notation omitting superscript $(k)$  meaning (\ref{eq pp^* without projections pr}):
	\begin{equation}\label{eq series for p^*}
	\begin{split}
		\pp^*(x_a) =  y_a^0+ y_a^2+ y_a^4+\cdots; \quad  
	\pp^*(\xi_b) =  \eta^1_b + \eta^3_b + \eta^5_b + \cdots .
	\end{split}
	\end{equation}
	Note that the series on the right hand side (\ref{eq series for p^*}) are not elements of the structure sheaf $\mcO_{\mcP}$. 
\end{remark}

\begin{remark}\label{rem P_n to M}
	Note that if the morphism $\pp^{(k)} : \mcP^{(k)}\to \mcM$, $k\geq 0$, there exists, it is locally determined by Formula (\ref{eq pp^* without projections pr}). Indeed, by Formula \ref{eq iso of sheaves} the morphism  $\pp^{(k)}$ can be regarded as a morphism $\pp^{(k)}: \mcP^{(k)}_k\to \mcM$, which is determined by the images of local coordinates, see Theorem \ref{theor Leites analog psi is unique}.

	 Further, if we have (\ref{eq pp^* without projections pr}), by Theorem \ref{theor Leites analog psi is existence} there exists unique morphism 
	 $$
	 \pp^{(k)}: \mcV^{(k)}_{\lambda}\simeq (\mcV_{\lambda})_k^{(k)}\to \mcU_{\lambda}, \quad k\geq 0,
	 $$ 
	 compatible with (\ref{eq pp^* without projections pr}). 
\end{remark}

\subsection{Properties of a $\Z^{\geq 0}$-covering of a supermanifold}
We will prove that our covering, if exists, satisfies the universal properties as other coverings in geometry and algebra. 
Let as above $\mcM=(\mcM_0,\mcO)$ be a supermanifold and $\mcN=(\mcN_0,\mcO_{\mcN})$ be a graded manifold. 
Denote $\mcN^{(k)} := (\mcN_0, \mcO_{\mcN}/(\mcI_{\mcN})_k)$, where $(\mcI_{\mcN})_k$ is the sheaf of ideals in $\mcO_{\mcN}$ generated by $(\mcO_{\mcN})_q$ for any $q\geq k+1$.

\begin{theorem}\label{theor univ_property supermanifold} Let $\mcN$ be a graded manifold, $\mathcal M$ be a supermanifold, $(\psi^{(k)}: \mcN^{(k)} \to \mcM)$ be a morphism see Definition \ref{def morphism N to M} and let $\mathcal P$  be a $\Z^{\geq 0}$-covering of $\mathcal M$, see Definition \ref{def new-covering of supermanifold}. Then there exists unique morphism of graded manifolds $\Psi= \varprojlim \Psi_k: \mcN\to \mcP$, where $\Psi_k: \mcN\to \mathcal P_k$ such that the following diagrams are commutative for any $k\geq 0$
	\begin{equation}\label{eq commut diagramm univ prop}	\begin{tikzcd}
	& \mathcal P^{(k)}  \arrow[d, "\pp^{(k)}"]\\
	\mcN^{(k)} \arrow[ur, dashrightarrow, "\exists ! \Psi^{(k)}"]   \arrow[r, "\psi^{(k)}"]  & \mathcal M,
	\end{tikzcd}
	\end{equation}
where $\Psi^{(k)}: \mcN^{(k)} \to \mathcal P^{(k)}$ are morphisms of ringed spaces naturally induced by $\Psi$.  
\end{theorem}

\begin{proof} We split the proof into steps. 
	
	{\bf Step 1.} Recall that any sheaf morphism $\psi^{(k)}$ has the same underlying map $\psi_0$, see Definition \ref{def morphism N to M}. Further by Definition \ref{def new-covering of supermanifold} we can choose atlases $\{\mathcal U_{\lambda}\}$ and $\{\mathcal V_{\lambda} = \varprojlim (\mathcal V_{\lambda})_n \}$ on $\mcM$ and $\mcP$, respectively.
	We put $W_{\lambda}:= \psi^{-1}_0(( \mathcal U_{\lambda})_0)$ and $\mathcal W_{\lambda}:=(W_{\lambda},\mcO_{\mcN}|_{W_{\lambda}}) $. Let us define the morphism $\Psi_{W_{\lambda}}: \mathcal W_{\lambda}\to \mcV_{\lambda}$ for any ${\lambda}$ using coordinates $(x_a,\xi_b)$ in $\mcU_{\lambda}$ and adapted coordinates $(y^t_a,\eta^s_b)$  in $\mcV_{\lambda}$, see Definition \ref{def new-covering of supermanifold}. To define $\Psi_{W_{\lambda}}$ we define first 
	$$
	(\Psi_{W_{\lambda}})_n :\mathcal W_{\lambda}\to (\mcV_{\lambda})_n \quad  \text{for any} \quad  n\geq 0.
	$$ 
	 We put 
	\begin{align*}
		&((\Psi_{W_{\lambda}})_n)^* (y_a^s ) := \pr_{s} \circ (\psi^{(n)})^* ( x_a) \in (\mcO_{\mcN})_s|_{W_{\lambda}}, \\ &((\Psi_{W_{\lambda}})_n)^* (\eta^t_b ) := \pr_{t} \circ (\psi^{(n)})^* ( \xi_b) \in (\mcO_{\mcN})_t|_{W_{\lambda}}, \quad s,t\leq n,
	\end{align*}
	where $\pr_q:\mathcal O_{\mcN} \to (\mathcal O_{\mcN})_q\subset \mcO_{\mcN}/(\mcI_{\mcN})_n$, $q\leq n$, is the natural projection. By \cite[Section 2.1.7, Theorema]{Leites} 
	we defined a morphism $(\Psi_{W_{\lambda}})_n: \mathcal W_{\lambda} \to( \mathcal V_{\lambda})_n$ of graded domains (considered as superdomains) for any $n$. The family $((\Psi_{W_{\lambda}})_n)$ is compatible by construction, hence it defines the morphism  $\Psi_{W_{\lambda}}: \mathcal W_{\lambda} \to \mathcal V_{\lambda}$, where 
	$$
	\Psi_{W_{\lambda}} := \varprojlim (\Psi_{W_{\lambda}})_n.
	$$ 
	By construction, $(\Psi_{W_{\lambda}})^*$ is $\Z$-graded.

	{\bf Step 2.} Let us prove that the morphism $\Psi_{W_{\lambda}}$ satisfies the following equality
	\begin{equation}\label{eq Psi|W_i}
	(\Psi_{W_{\lambda}}^{(k)})^* \circ (\pp^{(k)})^*(F) =  (\psi^{(k)})^*(F),\quad k\geq 0,
	\end{equation}
	where  $F\in \mathcal O|_{(\mathcal U_{\lambda})_0}$.  Indeed,  by definition of $\Psi_{W_{\lambda}}$ the equality (\ref{eq Psi|W_i}) holds for $F=x_a$ or $\xi_b$. Now we apply Theorem \ref{theor Leites analog psi is unique}.

	{\bf Step 3.} We have to show that  $\Psi_{W_i} = \Psi_{W_j}$ in $W_i\cap W_j$. Let $(\tilde x_{c}, \tilde \xi_{d})$ be local coordinates in $\mathcal W_j$, $\tilde x_{c}= H(x_a,\xi_b)$ in $\mathcal U_i\cap \mathcal U_j$  and $\tilde y^s_{c}= \pr_{s} \circ (\pp^{(k)})^* ( \tilde x_{c})$, $k\geq s$. We have by Step $2$ for $k\geq s$ in $W_i\cap W_j$:
	\begin{align*}
	(\Psi_{W_j}^{(k)})^* (\tilde y^s_{c}) = \pr_s\circ (\Psi_{W_j}^{(k)})^* &( (\pp^{(k)})^*(\tilde x_{c})) = \pr_s\circ (\psi^{(k)})^*(\tilde x_{c}) =\\ 
	\pr_s\circ (\psi^{(k)})^*(H(x_a,\xi_b))=& 
  (\Psi_{W_i}^{(k)})^*  (\pr_s\circ (\pp^{(k)})^* (H(x_a,\xi_b))) = (\Psi_{W_i}^{(k)})^* (\tilde y^s_{c}).
	\end{align*}
	Similarly we can prove that $(\Psi_{W_j}^{(k)})^* (\tilde \eta^t_{d}) = (\Psi_{W_i}^{(k)})^* (\tilde \eta^t_{d})$, where $ \tilde \eta^t_{d} = \pr_{t} \circ (\pp^{(k)})^* ( \tilde \xi_{d})$,  $k\geq t$. Now the statement follows from Theorem \ref{theor Leites analog psi is unique}. We define the morphism $\Psi$ by $\Psi|_{W_i}:= \Psi_{W_i}$.

	{\bf Step 4.} The family $\Psi^{(k)}$ is unique. Indeed, above, see Step 1, we saw that locally there is unique way to make the Diagram \ref{eq commut diagramm univ prop} commutative. The proof is complete. 
\end{proof}

\begin{remark}
To shorten the notation, we will sometimes write 
	$$
	\begin{tikzcd}
	& \mathcal P \arrow[d, "\pp"]\\
	\mcN \arrow[ur, dashrightarrow, "\exists ! \Psi"]   \arrow[r, "\psi"]  & \mathcal M
	\end{tikzcd}
	$$
	instead of (\ref{eq commut diagramm univ prop}). 
\end{remark}

 \begin{corollary}\label{corol transition function for a covering}
 	In notations of the proof of Theorem  \ref{theor univ_property supermanifold}, Step 3, let $\tilde x_c= H_c(x_a,\xi_b)$ and $\tilde \xi_d= G_d(x_a,\xi_b)$ be transition functions in $\mathcal U_i\cap \mathcal U_j$. By definition of adapted coordinates we have $\tilde y^s_c= \pr_{s} \circ \pp^* ( \tilde x_c) $ and $\tilde \eta^t_d= \pr_{t} \circ \pp^* ( \tilde \xi_d) $. Then in $\mathcal V_i\cap \mathcal V_j$ we have
 	$$
 	\tilde y^s_c = \pr_{s} \circ \pp^*  (H_c(x_a,\xi_b)),\quad \tilde \eta^t_d = \pr_{t} \circ \pp^*  (G_d(x_a,\xi_b)). 
 	$$
 \end{corollary}
 
 \begin{proof}
 	The statement follows from the definition of adapted coordinates. Indeed, 
 	\begin{align*}
 	\tilde y^s_c = \pr_{s} \circ \pp^* ( \tilde x_c)=  \pr_{s} \circ \pp^*  (H_c(x_a,\xi_b)),\quad \tilde \eta^t_d = \pr_{t} \circ \pp^* ( \tilde \xi_d)= \pr_{t} \circ \pp^*  (G_d(x_a,\xi_b)). 
 	\end{align*} 
 \end{proof}
 
 \begin{remark}\label{rem p_n:mcP^(n) to mvM}
Note that if the covering morphism $\pp: \mcP\to \mcM$ there exists, this does not imply that there exists an induced morphism $\pp_n: \mcP_n\to \mcM$. However we have the naturally induced morphism $\pp^{(n)}: \mcP^{(n)}_n\to \mcM$ for any $n$. Further let $\mcN=(\mcN_0,\mcO_{\mcN})$ be a graded manifold of degree $n$ and  $\mcN^{(n)}=(\mcN_0,\mcO_{\mcN}/(\mcI_{\mcN})_n)$ be as above. Let we have a morphism of $\Z_2$-graded locally ringed spaces $\psi^{(n)}:\mcN^{(n)} \to \mcM$. (Note that $\psi^{(n)}$ is not necessary induced by a morphism  $\mcN \to \mcM$.) In this case we also can find unique morphism  $\Psi^{(n)}:\mcN^{(n)} \to \mcP^{(n)}$ such that $\pp^{(n)} \circ \Psi^{(n)} = \psi^{(n)}$. To do this we repeat the argument of the proof of Theorem \ref{theor univ_property supermanifold}. More precisely, in Step $1$, Theorem \ref{theor univ_property supermanifold}, we replace $(\Psi_{W_{\lambda}})_n$ by $(\Psi_{W_{\lambda}}^{(n)})_n$ and repeat word to word the rest.  
\end{remark}

Now we are ready to prove the following statement.

\begin{theorem}\label{theor morphism psi can be covered by Psi}
	Let $\mcM$, $\tilde{\mcM}$ be two supermanifolds and $\mathcal P$, $\tilde{\mathcal P}$ be their $\Z^{\geq 0}$-coverings, respectively. Let $\psi:\mcM\to \tilde{\mcM}$ be a morphism of supermanifolds. Then there exists  unique morphism $\Psi:\mathcal P\to \tilde{\mathcal P}$ of $\Z$-graded manifolds such that the following diagram is commutative
	$$
	\begin{tikzcd}
	\mathcal P \arrow[r, "\exists!\Psi"] \arrow["\pp", d]
	& \tilde{\mathcal P} \arrow[d, "\tilde\pp" ] \\
	\mcM\arrow[r,  "\psi" ]
	& \tilde{\mcM}
	\end{tikzcd}
	$$
	In particular, if a $\Z^{\geq 0}$-covering of a supermanifold $\mcM$ there exists, it is unique up to isomorphism. 
\end{theorem}

\begin{proof}
	Consider the morphism $(\psi\circ \pp)^{(k)}: \mcP^{(k)} \to \tilde\mcM$. Recall that $ \mcP^{(k)} \simeq  \mcP^{(k)}_n$, where $n\geq k$.  By Remark \ref{rem p_n:mcP^(n) to mvM} there exists unique morphism $\Psi^{(k)} : \mcP^{(k)} \to \tilde\mcP^{(k)}$ for any $k\geq 0$. Let us show that the family $(\Psi^{(k)})$ is compatible. Indeed,  if 
	$$
	\pr^{k+1}_k: \mcO_{\mcP}/\mcI_{k+1} \to \mcO_{\mcP}/\mcI_{k}
	$$ 
	is the natural morphism, then both morphisms $(\Psi^{(k)})^*$ and $\pr^{k+1}_k\circ (\Psi^{(k+1)})^*$ make Diagram (\ref{eq commut diagramm univ prop}) commutative. Hence, 
	$$
	(\Psi^{(k)})^* = \pr^{k+1}_k\circ (\Psi^{(k)})^*.
	$$

	Further, since any $(\Psi^{(k)})^*$ is $\Z$-graded, we have 
	$$
	(\Psi^{(k)})^* \big((\mcO_{\tilde{\mathcal P}})_n\big) \subset (\mcO_{\mathcal P})_n, \quad n\leq k. 
	$$
Since $(\mcO_{\tilde{\mathcal P}})_n = (\mcO_{\tilde{\mathcal P}_n})_n $	and $(\mcO_{{\mathcal P}})_n = (\mcO_{{\mathcal P}_n})_n $ for any $n\geq 0$, we have 
\begin{equation}\label{eq tecjnical 2}
(\Psi^{(k)})^* \big((\mcO_{\tilde{\mathcal P}_n})_n\big) \subset (\mcO_{\mathcal P_n})_n, \quad n\leq k.
\end{equation}
Let us take a local coordinate $\zeta^r$ in $\mcO_{\tilde{\mathcal P}_r}$ of degree $r$. By (\ref{eq tecjnical 2}), we see that $(\Psi^{(k)})^* (\zeta^r) \subset \mcO_{\mathcal P_r^{(k)}}$, where $r\leq k$. Since $r$ is any, we get
$$
(\Psi^{(k)})^* \big(\mcO_{\tilde{\mathcal P}^{(k)}_n}\big) \subset \mcO_{\mathcal P_n^{(k)}}, \quad \text{for any}\,\,\, k. 
$$
In other words, we get a compatible family of morphisms $(\Psi^{(k)}: \mcP_n^{(k)} \to \tilde\mcP_n^{(k)})$ of $\Z$-graded ringed spaces. Such a family is induced by a morphism $\Psi_n: \mcP_n \to \tilde\mcP_n$. Now we put $\Psi= \varprojlim \Psi_n$.

To show uniqueness of $\mcP$, we put $\tilde \mcM=\mcM$ and $\psi=\id$. 
\end{proof}

Let us study the universal property (\ref{eq commut diagramm univ prop}) in more details. Denote by $\tilde \mcP$ a graded manifold, if there exists, satisfying the universal property 
(\ref{eq commut diagramm univ prop}). Let us prove that that satisfy Definition \ref{def new-covering of supermanifold}. 

\begin{theorem}\label{theor def of cov using univ prop}
If there exists a graded manifold 	$\tilde \mcP$ together with a morphism $\tilde \pp:\tilde \mcP \to \mcM$  satisfying the universal property  (\ref{eq commut diagramm univ prop}), then it satisfies Definition \ref{def new-covering of supermanifold}. 
\end{theorem}

\begin{proof}
	Let $\{\mcU_{\al}\}$ be an atlas on $\mcM$. Denote 
	$$
	\tilde \mcV_{\al}:= (\tilde\pp^{-1}_0((\mcU_{\al})_0),  \mcO_{\tilde \mcP}|_{\tilde\pp^{-1}_0((\mcU_{\al})_0)}). 
	$$
For any superdomain $\mcU_{\al}$ by Theorem \ref{theor construction of a covering domain} we can construct its covering $\pp:\mcV_{\al}\to \mcU_{\al}$. Consider the following morphisms $\pp_k:(\mcV_{\al})_{2k+1}\to \mcU$
\begin{align*}
&	\pp_{k}^*(x_a) =  y_a^0+ y_a^2+ y_a^4+\cdots + y_a^{2k}; \\  
&	\pp_{k}^*(\xi_b) =  \eta^1_b + \eta^3_b + \eta^5_b + \cdots  + \eta_b^{2k+1},
\end{align*}
where $(x_a,\xi_b)$ are coordinates of $\mcU_{\al}$. 
We can apply the universal property of $\tilde \mcP$ to any $\pp_k$ and obtain the following commutative diagrams
$$
\begin{tikzcd}
&\tilde \mcV_{\al} \subset \tilde\mcP^{(k)}  \arrow[d, "\tilde\pp^{(k)}"]\\
(\mcV_{\al})^{(k)}=(\mcV_{\al})_{2k+1}^{(k)} \arrow[ur, dashrightarrow, "\exists ! P^{(k)}"]   \arrow[r, "\pp_k^{(k)}"]  & \mathcal M,
\end{tikzcd}
$$
The family $P^{(k)}$, $k\geq 0$, determines a morphism $P:\mcV_{\al}\to\tilde \mcV_{\al}$ of graded domains. Further we use the fact that $\pp:\mcV_{\al}\to \mcU_{\al}$ satisfies the universal property as well, see Theorem \ref{theor univ_property supermanifold}, and repeat a similar argument. We get a morphism $Q: \tilde \mcV_{\al} \to \mcV_{\al}$.  Now we see that $Q\circ P:\mcV_{\al} \to \mcV_{\al}$ is a cover of $\id:\mcU_{\al} \to \mcU_{\al}$, and the same for the graded morphism $P\circ Q: \tilde \mcV_{\al} \to \tilde \mcV_{\al}$. Therefore, $Q\circ P=\id$ and $P\circ Q=\id$, hence $\tilde \mcV_{\al} \simeq \mcV_{\al}$.
\end{proof}

 We finalize this section with the following observation.  

\begin{remark} Let $\mcP$ be a covering of a supermanifold $\mcM$ of dimension $m|n$. Definition \ref{def new-covering of supermanifold} implies that in some sense $\mathcal P$ is locally diffeomorphic to $\mcM$. Indeed, let us choose $q\in \Zo$ assuming for example that $q$ is even. Then the sheaf morphism $\pr_q\circ\pp^*$ determines a sheaf isomorphism of $S^*(x_1,\ldots, x_m)|_{(\mathcal U_i)_0}\subset \mathcal F_{\mcM_0}$ and $S^*(y^q_1,\ldots, y^q_m)|_{(\mathcal U_i)_0}$. Here $S^*(z_1,\ldots, z_m)$ is the supersymmetric algebra generated by $z_1,\ldots, z_m$. And similarly in the case of odd coordinates. In other words we have the following isomorphism of superdomains
	\begin{align*}
	((\mathcal U_i)_0, S^*(x_a)) \simeq  ((\mathcal V_i)_0, S^*(y^s_a)), \quad ((\mathcal U_i)_0, S^*(\xi_j)) \simeq  ((\mathcal V_i)_0, S^*(\eta^t_b)),
	\end{align*}
	where we denoted for brevity $S^*(z_a) =S^*(z_1,\ldots, z_m)$.
\end{remark}

\section{A construction of the universal covering of a superdomain}

\subsection{A covering for a superdomain}
To construct a covering space for a superdomain $\K^{p|q}$ we use Definition \ref{def new-covering of supermanifold}.  Let $(x_a,\xi_b)$ be a coordinate system in a superdomain $\mcU\subset \K^{p|q}$. By Definition \ref{def new-covering of supermanifold} a covering space of $\mathcal U\subset \K^{p|q}$ is a graded domain $\mcV$ with the base space $\mcV_0=\mcU_0$ and with coordinates $(y_a^t,\eta_b^s)$, where $s$ is an even integer and  $t$ is an odd integer. (We assume that the underlying space $\mcU_0$ is sufficiently small.) The covering maps $(\pp^{(k)})$ are given by (\ref{rem series for p^*}), see also (\ref{eq series for p^*}), in local coordinates. By Theorem \ref{theor Leites analog psi is existence}, we get a compatible family $(\pp^{(k)}: \mcV^{(k)} \to \mcM)$.  More precisely, let us take $F\in \mcO_{\K^{p|q}}$ than for any $k\geq 0$ we have
\begin{equation}\label{eq p^*_k(F) Taylor series}
\begin{split}
(\pp^{(k)})^* (F(x_a,\xi_b)) =
F(y^0_a +y_a^2+\cdots + y_a^{2v}, \eta_b^1+\eta_b^3 + \cdots + \eta_b^{2v+1} ) \mod \mathcal I_k,
\end{split}
\end{equation}
where $v>k$ is any integer. The right hand side of (\ref{eq p^*_k(F) Taylor series}) is defined by Taylor series.

Denote by $\mathcal V_n$ the graded domain with coordinates $(y_a^s,\eta^t_b)_{s\leq n,t\leq n}$. We summarize all of the above in the following theorem. 

\begin{theorem}\label{theor construction of a covering domain}
The graded manifold $\mcV= \varprojlim \mcV_n$ together with the family of morphisms $\pp^{(k)} : \mcV^{(k)}\to \mcU$, where $k\in \Z^{\geq 0}$, defined by Formula (\ref{eq p^*_k(F) Taylor series})  satisfy Definition \ref{def new-covering of supermanifold}. In other words, $\pp:\mcV\to \mcU$ is a $\Z^{\geq 0}$-covering of $\mcU$. 
\end{theorem}

\subsection{Examples}
Let us consider some examples. We start with the simplest non-trivial supermanifold.

\begin{example}
	Consider the case of pure odd supermanifold $\mcM=(\pt, \bigwedge(\xi))$ of dimension $0|1$. In this case $\mcV = (\pt, \bigwedge(\eta^1,\eta^3,\ldots)$ and we have
	$$
	\pp^*(F) = F(\pt) + \frac{\partial  F}{\partial \xi}(\pt) \eta^1 + \frac{\partial  F}{\partial \xi}(\pt) \eta^3 + \cdots . 
	$$
	This formula implies that $\pp^*(const) =const $ and $\pp^*(\xi) = \eta^1+\eta^3+\cdots$. 
\end{example}

\begin{example}
	Consider the case of pure odd two dimensional supermanifold $\mcM=(\pt, \bigwedge(\xi_1,\xi_2))$. In this case $\mcV = (\pt, \bigwedge(\eta^1_1,\eta^3_1,\ldots, \eta^1_2,\eta^3_2,\ldots))$  and we have
	$$
	\pp^*(F) = F(\pt) + \sum_{i=1,2}\sum_{q=0}^{\infty}  \frac{\partial  F}{\partial \xi_i}(\pt) \eta^{2q+1}_i +  \sum_{q=1}^{\infty} \sum_{q_1+q_2= 2q}\frac{\partial^2  F}{\partial \xi_1 \partial \xi_2}(\pt)\eta^{q_1}_1 \eta^{q_2}_2,
	$$
	where $q_1$ and $q_2$ are odd non-negative integers and $F\in \bigwedge(\xi_1,\xi_2)$. 
	This formula implies that $\pp^*(const) =const$ and $\pp^*(\xi_i) = \eta^1_i+\eta^3_i+\eta^5_i+\cdots$ for $i=1,2$. 
\end{example}

\begin{example}
	Consider the case of pure even $1|0$-dimensional supermanifold $\mcM=\mathbb R$ or $\mathbb C$. In other words $\mcM$ is a usual $1$-dimensional manifold. Clearly $\mcM$ can be covered by one chart with local coordinate $x$. In this case a $\Zo$-covering $\mcV$ is the graded domain with graded coordinates  $(y^s)$, where $s\geq 0$ is even. In this case we can rewrite (\ref{eq p^*_k(F) Taylor series}) in the following way for any $F\in \mcO_{\mcM}$
	\begin{align*}
	\pp^*(F(x)) = &F(y^0+y^2+y^4+\cdots ) =\\
	 &=F(y^0) + \frac{\partial F}{\partial x}(y^0) (y^2+ y^4+\cdots ) +
	 \frac{1}{2!} \frac{\partial^2 F}{\partial x^2}(y^0) (y^2+ y^4+\cdots )^2 + \cdots .
	\end{align*}
  For instance,
	\begin{align*}
	&\pr_0\circ\pp^*(F) =F(y^0); \\
	&\pr_2\circ\pp^*(F) =  \frac{\partial F}{\partial x}(y^0) y^2;\\
	&\pr_4\circ\pp^*(F) = \frac{\partial F}{\partial x}(y^0) y^4 + \frac{1}{2!} \frac{\partial^2 F}{\partial x^2}(y^0) (y^2)^2;\\
	&\pr_6\circ\pp^*(F) =  \frac{\partial F}{\partial x}(y^0) y^6 + \frac{\partial^2 F}{\partial x^2} y^2y^4 + \frac{1}{3!}  \frac{\partial^3 F}{\partial x^3}(y^0) (y^2)^3 ;\\
	&\cdots .
	\end{align*}
For instance wee see that $\pp^*(x) = y^0+y^2+y^4+\cdots$. Later we will give a geometric description of this graded manifold.
\end{example}

\begin{example}
Consider the case of $1|2$-dimensional supermanifold $\mcM=(\K, \mcF_{\K}\otimes\bigwedge(\xi_1,\xi_2))$, where $\K=\R$ or $\C$. Let us describe first homogeneous terms. We have for any $F\in \mcO_{\mcM}$
\begin{align*}
\pr_0\circ\pp^*(F(x_a,\xi_b)) =&F(y^0_a,0); \\
\pr_1\circ\pp^*(F(x_a,\xi_b)) = & \frac{\partial F}{\partial \xi_1}(y^0_a,0) \eta^1_1+ \frac{\partial F}{\partial \xi_2}(y^0_a,0) \eta^1_2;\\
\pr_2\circ\pp^*(F(x_a,\xi_b)) = &\frac{\partial F}{\partial x}(y^0_a,0) y^2 + \frac{\partial^2 F}{\partial \xi_1 \partial \xi_2}(y^0_a,0) \eta^1_1 \eta^1_2;\\
\pr_3\circ\pp^*(F(x_a,\xi_b)) =  &\frac{\partial F}{\partial \xi_1}(y^0_a,0) \eta_1^3 + \frac{\partial F}{\partial \xi_2}(y^0_a,0) \eta_2^3 + 
 \frac{\partial^2 F}{\partial x \partial \xi_1}(y^0_a,0) y^2\eta^1_1  +\\ 
 &\frac{\partial^2 F}{\partial x \partial \xi_2}(y^0_a,0) y^2\eta^1_2 ;\\
&\cdots .
\end{align*}
\end{example}

\subsection{Functorial properties of $\Zo$-coverings for superdomains}

If $\mathcal U$ is a superdomain with coordinates $(x_a,\xi_b)$, we denote $\F(\mathcal U):=\mcV$ the graded domain constructed above with coordinates $(y_a^t,\eta_b^s)$, where $s$ is an even non-negative integer and  $t$ is an odd non-negative integer.

Further let $\psi:\mathcal U \to  \mathcal U'$ be a morphism of superdomains. By Theorem \ref{theor morphism psi can be covered by Psi} there exists unique $\Psi: \mcV\to \mcV'$ such that $\pp'\circ \Psi = \psi\circ \pp$, where $\pp': \mcV'\to \mcU'$ is a $\Zo$-covering of $\mcU'$. We put $\F(\psi):=\Psi$. Summing up, we constructed a mapping $\F$ from the category of superdomains to the category of graded domains $\mcV= \varprojlim \mcV_n$. If $\psi=\id$, then clearly $\Psi=\id$.

\begin{example}[Coordinate description of $\F(\psi)$]
Assume that the superdomain $\mathcal U$ and $\mathcal U'$ have  coordinates $(x_a,\xi_b)$ and $(u_c,\theta_d)$, respectively. Let $\mcV= \F(\mathcal U)$ with coordinates $(y_a^t,\eta_b^s)$ and  $\mcV'= \F(\mathcal U')$ with coordinates $(z_c^t,\zeta_d^s)$ be as above. Any morphism $\psi:\mathcal U \to  \mathcal U'$ is determined by the images of local coordinates:
	\begin{align*}
	\psi^*(u_c) = F_c(x_a,\xi_b),\quad \psi^*(\theta_d) = G_d(x_a,\xi_b). 
	\end{align*}
	 For the morphism $\Psi:\mathcal V \to  \mathcal V'$ in coordinates we have
	\begin{equation}\label{eq F of a morphism in coordinates}
	\begin{split}
	&z_c^t = \pr_t\circ \pp^*(\psi^*(u_c) ) = \pr_t \circ\pp^*(F_c(x_a,\xi_b)); \\ 
	&\zeta_d^s = \pr_s\circ \pp^*(\psi^*(\theta_d) ) = \pr_s \circ\pp^* (G_d(x_a,\xi_b)).
	\end{split}
	\end{equation}
\end{example}

\begin{theorem}\label{theor F(Psi circ Phi) = F(Psi) circ F(Phi)}
Let we have three superdomains $\mcU$, $\mcU'$ and $\mcU''$ and two moprphisms $\psi:\mcU \to \mcU' $  and $\psi':\mcU' \to \mcU''$. Then we have
$$
\F(\psi' \circ \psi) = \F(\psi') \circ\F (\psi).
$$
\end{theorem}

\begin{proof} 
The statement follows from Theorem 	\ref{theor morphism psi can be covered by Psi}. Indeed, $\F(\psi') \circ\F (\psi)$ and $\F(\psi' \circ \psi) $ both cover $\psi'\circ\psi$. Since the covering of $\psi'\circ\psi$ in unique by Theorem 	\ref{theor morphism psi can be covered by Psi}, we get the result. 
\end{proof}

\begin{corollary}
	The mapping $\F$ is a functor from the category of superdomains to the category of graded domains $\mcV= \varprojlim \mcV_n$. 
\end{corollary}

\begin{proof}
	Apply Theorem \ref{theor F(Psi circ Phi) = F(Psi) circ F(Phi)}. 
\end{proof}

\section{A construction of a $\Z^{\geq 0}$-covering of a supermanifold}

The construction is based on Definition \ref{def new-covering of supermanifold} and a construction of a $\Z^{\geq 0}$-covering for superdomain. The idea is to glue together $\Z^{\geq 0}$-coverings for superdomains and obtain a graded manifold. 

\begin{theorem}\label{theor construction of a covering for any supermanifld}
	For any supermanifold $\mcM$  over $\K=\R$ or $\C$ there exists  a $\Zo$-covering $\pp: \mcP\to \mcM$ of $\mcM$. 
\end{theorem}

\begin{proof}
Let $\{\mcU_i\}$ be an atlas on $\mcM$. By Theorem \ref{theor construction of a covering domain} for any superdomain $\mcU_i$ there exists a $\Zo$-covering $\pp_i:\mcV_i \to \mcU_i$, where $\mcV_i= \F(\mcU_i)$. Consider three charts $\mathcal U_1, \mathcal U_2$ and $\mathcal U_3$ on $\mcM$ with $\mathcal U_1\cap \mathcal U_2\cap \mathcal U_3\ne \emptyset$. Denote by $T_{ij}$ the transition function  $T_{ij}: \mathcal U_i\to \mathcal U_j$.  By Theorem \ref{theor F(Psi circ Phi) = F(Psi) circ F(Phi)} we have
	$$
	\F(T_{31})\circ \F(T_{23})\circ \F( T_{12}) = \F(T_{31}\circ T_{23} \circ  T_{12} ) = \F(\id) =\id.
	$$
Therefore we get a cocycle $\{\F(T_{ij}) \}$ and we can define the graded manifold $\mcP$ using graded domains $\mcV_i$ and transition functions $\F(T_{ij})$. 

It is remaining to define a covering morphism $\pp:\mcP\to \mcM$. Let us prove that  $\pp_i|_{\mcU_i\cap \mcU_j} = \pp_j|_{\mcU_i\cap \mcU_j}$. Indeed,  by Theorem \ref{theor morphism psi can be covered by Psi} and Theorem \ref{theor construction of a covering domain} we have 
$$
T_{ij}\circ\pp_i = \pp_j \circ \F(T_{ij}). 
$$
Now we put $\pp|_{\mcU_i} = \pp_i$ for any $i$. The proof is complete.
\end{proof}

We defined a mapping $\F$, $\mcM \mapsto \mcP$, from the category of supermanifolds to the category of graded manifolds $\mcN= \varprojlim \mcN_n$, see Section \ref{sec inverse limit}. 

\begin{theorem}
	The mapping $\F$ is a functor from the category of supermanifolds to the category of graded manifolds $\mcN= \varprojlim \mcN_n$.  
\end{theorem}

\begin{proof}
By Theorem \ref{theor construction of a covering for any supermanifld} to any supermanifold $\mcM$ we can assign a $\Zo$-covering $\mcP$ of $\mcM$. We put $\F(\mcM) = \mcP$ as constructed in Theorem \ref{theor construction of a covering for any supermanifld}. Further similarly to the proof of Theorem \ref{theor F(Psi circ Phi) = F(Psi) circ F(Phi)} we get that $\F(\psi_1\circ \psi_2) = \F(\psi_1)\circ \F(\psi_2)$, where $\F(\psi) =\Psi$ is as in Theorem \ref{theor morphism psi can be covered by Psi}.  
\end{proof}

Let $\mcN = \varprojlim \mcN_n$ and we put $\pr_n: \mcN\to \mcN_n$ be the natural projection. Further to a morphism   $\Phi: \mcN \to \mcN'$, where $\mcN= \varprojlim \mcN_k$ and $\mcN'= \varprojlim \mcN'_k$, we can assign the morphism $\Phi_k: \mcN_k \to \mcN'_{k}$, see Section \ref{sec inverse limit}.  Now for any supermanifold $\mcM$, we put $\F_n(\mcM) := \mcP_n$, where $\mcP$ is a $\Zo$-covering of $\mcM$ constructed in Theorem \ref{theor construction of a covering for any supermanifld}, and we put $\F_n(\psi) := \F(\psi)_n= \Psi_n$ for a morphism of supermanifolds $\psi$.  Let us summarize this observations in the following proposition.

\begin{proposition}\label{prop F_n is a functor}
	The mapping $\F_n$, where $n\geq 0$, is a functor from the category of supermanifolds to the category of graded manifolds of degree $n$.  
\end{proposition}

\section{A geometric realization of a $\Z^{\geq 0}$-covering of a supermanifold}  Let us remind first a construction of the $k$-th-order tangent bundle $T^k(M)$ of a usual manifold $M$.

\subsection{$k$-th-order tangent bundle of a manifold $M$}
Let $M$ be a smooth or holomorphic manifold and $p\in M$. Consider two smooth functions $f,g:W \to M$, where $W$ is a neighborhood of $0\in \K$,  such that $f(0)=g(0) = p$. We say that $f$ and $g$ are {\it equivalent to order $k$ at $p$} if there is some neighborhood $U$ of $p$, such that, for every function  $\varphi :U\to \K$ the $k$-th-order jets $J^k_0(\varphi\circ f)$, $J^k_0(\varphi\circ g)$ at point $0$ are equal, that is  $J^k_0(\varphi\circ f)=J^k_0(\varphi\circ g)$.  Denote the $k$-jet of a curve $f$ through $p$ to be the equivalence class of $f$ under this equivalence relation. The $k$-th-order jet space $J^k_0(\K,M)_p$ is then the set of $k$-jets at $p$. The {\it $k$-th-order tangent bundle $T^k(M)$} of a manifold $M$ is a fiber bundle with the fiber $J^k_0(\K,M)_p$ ap $p\in M$. If $(x_a)$ are local coordinates in $U$ at $p$, the standard local coordinates in $T^k(U) \subset T^k(M)$ are 
$$
(x_a, \dot{x}_a, \ddot{x}_a, \ldots, x^{(k)}_a),
$$
where 
\begin{align*}
x^{(s)}_a := \frac{\partial^s}{\partial t^s}\Big|_{t=0} (x_a \circ f(t)), \quad f:W \to U 
\end{align*}
For instance, $T^1(M) = T(M)$. Denote $T^{\infty}(M):=\varprojlim T^k(M)$.

\begin{example}
Let	$(x_a)$ be local coordinates in $U$ and $y_c= F_c(x_a)$ be a coordinate change. Then the transformation law for coordinates $y_c, \dot{y}_c, \ddot{y}_c$ is the following
\begin{align*}
y_c= F_c(x_a), \quad \dot{y}_c = \sum_i\frac{\partial F_c}{\partial x_i} \dot{x}_i, \quad  \ddot{y}_c = \sum_i\frac{\partial F_c}{\partial x_i} \ddot{x}_i + \sum_{ij}\frac{\partial^2 F_c}{\partial x_i \partial x_j} \dot{x}_i \dot{x}_j. 
\end{align*} 
\end{example}

The fiber bundle $T^k(M)$ is a graded manifold $(M, \mcO_{T^k(M)})$ of degree $k$ assuming that $x^{(s)}_a$ has weight $s$. Denote by $\tilde\mcR$ the sheaf of ideals in $\mcO_{T^k(M)}$ locally generated by $x^{(2s+1)}_a$, where $s\geq 0$. Clearly $\tilde\mcR$ is well-defined.

\begin{proposition}\label{prop reduced O_P}
Let $M$ be a manifold and $T^{k}(M)$ be its 
$k$-th-order tangent bundle.   Then we have
$$
\mcO_{T^{2k}(M)}/\tilde\mcR \simeq \mcO_{T^{k}(M)}, \quad \mcO_{T^{\infty}(M)}/\tilde\mcR \simeq \mcO_{T^{\infty}(M)}. 
$$
\end{proposition}
\begin{proof}
In the standard local coordinates an isomorphism is given by 	$ x^{(2s)}_a \mapsto x^{(s)}_a$. 
\end{proof}

\subsection{The $k$-th-order tangent bundle of a supermanifold $\mcM$}

Let us briefly remind a construction  of the $k$-th-order tangent  bundle $T^k(\mcM)$ of a supermanifold $\mcM$, see 
 \cite{Bruce} for details. Denote by $\underline{\Hom}(\mcM',\mcM)$  the mapping supermanifold defined via functor of points by 
 $$
 \underline{\Hom}(\mcM',\mcM)(\mcS) := \Hom(\mcS\times \mcM',\mcM),
 $$
 where $\Hom(\mcS\times \mcM',\mcM)$ is the set of smooth or holomorphic morphisms $\mcS\times \mcM'\to \mcM$. 
 This is a generalized supermanifold, see for example \cite[Section: Generalized supermanifolds and the internal homs]{Bruce}. To define $T^k(\mcM)$, we need the following definition \cite[Definition 3]{Bruce}. 
 
 \begin{definition}
 	A curve on a supermanifold $\mcM$ parameterized by another supermanifold $\mcS$ is a morphism $ \gamma_{\mcS}\in \underline{\Hom}(\K,\mcM)(\mcS) = \Hom(\mcS\times \K,\mcM)$.  
 \end{definition}

Let us take $\Gamma_{\mcS}\in \underline{\Hom}(\K,\K^{1|1})(\mcS)$. The {\bf $k$-th jet of $\Gamma_{\mcS}$ at the point $t_0\in \K$} is the following polynomial with coefficients in $\mcO_{\mcS}$:
$$
J^k_{t_0} (\Gamma_{\mcS}) := \Gamma_{\mcS}\Big|_{t_0} + z\frac{\partial \Gamma_{\mcS}}{\partial t}\Big|_{t_0} + z^2 \frac{1}{2!} \frac{\partial^2 \Gamma_{\mcS}}{\partial t^2}\Big|_{t_0} + \cdots + z^k \frac{1}{k!} \frac{\partial^k \Gamma_{\mcS}}{\partial t^k}\Big|_{t_0}. 
$$
 Two elements $\Gamma_{\mcS}, \Gamma'_{\mcS}\in \underline{\Hom}(\K,\K^{1|1})(\mcS)$ are called {\bf equivalent to order $k$} if 
 $$
 J^k_{t_0} (\Gamma_{\mcS}) = J^k_{t_0} (\Gamma'_{\mcS}). 
  $$

  Let we have two curves $ \gamma_{\mcS},\gamma'_{\mcS}\in \underline{\Hom}(\K,\mcM)(\mcS)$. Two such curves are in {\bf contact to order $k$ at the $\mcS$-point  $x\in\Hom(\mcS, \mcM)$}, if 
  $$
  J^k_0(\varphi\circ \gamma_{\mcS})=J^k_0(\varphi\circ \gamma'_{\mcS}), \quad (\varphi\circ \gamma_{\mcS})|_0= (\varphi\circ \gamma'_{\mcS})|_0=x
  $$
  for any locally defined function $\phi: \mcM\to \K^{1|1}$. An equivalence class of this relation is called a $(k,\mcS)$-jet from $\K$ to $\mcM$ and it denoted by $[\gamma_{\mcS}]_{\mcS}^k$. Following \cite{Bruce}, we denote by $J^k_{0} (\K, \mcM (\mcS))_x$ the set of all $(k,\mcS)$-jets from $\K$ to $\mcM$ passing through the $\mcS$-point $x = \gamma_{\mcS}|_{\{0\}\times \mcS}:\mcS\to \mcM$.   And the set of all $(k,\mcS)$-jets from $\K$ to $\mcM$ we denote by  $J^k_{0} (\K, \mcM (\mcS))$.

 \begin{definition}\cite{Bruce}
 	The $k$-th-order tangent bundle $T^k(\mcM)$ of a supermanifold $\mcM$ is the (generalized) supermanifold defined by functor of points as follows.
 	$$
 	T^k(\mcM) (\mcS) := J^k_{0} (\K, \mcM (\mcS)). 
 	$$ 	
 \end{definition}
Note that $T^k(\mcM)$ is in fact a {\bf supermanifold}, see \cite[Theorem 1]{Bruce}. Moreover it is  a {\bf graded supermanifold} with local graded coordinates 
\begin{equation}\label{eq standart coordinates in T^k(mcM)}
(x_a, \dot{x}_a, \ddot{x}_a, \ldots, x^{(k)}_a,\xi_b, \dot{\xi}_b, \ddot{\xi}_b, \ldots, \xi^{(k)}_b),
\end{equation}
where the weight of $x^{(s)}_a$ and $\xi^{(s)}_b$ is equal to $s$. The coordinates $(x^{(s)}_a)$ are even, while the coordinates $(\xi^{(t)}_b)$ are odd.  Here we defined 
\begin{align*}
x^{(s)}_a := \frac{\partial^s}{\partial t^s}\Big|_{t=0} (x_a\circ \gamma_{\mcS}), \quad \xi^{(s)}_b := \frac{\partial^s}{\partial t^s}\Big|_{t=0} (\xi_b \circ \gamma_{\mcS}). 
\end{align*}

\begin{remark}\label{rem transition functions for  jets super}
Note that the transition functions for coordinates (\ref{eq standart coordinates in T^k(mcM)}) between standard local charts have the same form as in classical case, see \cite{Bruce}.  	
\end{remark}

Denote $T^{\infty}(\mcM):=\varprojlim T^k(\mcM)$.

\subsection{About graded supermanifolds of degree $k$}

Let we have  a graded supermanifold $\mcN=(\mcN_0,\mcO_{\mcN})$ of degree $k$ and of dimension $(m_q|n_q)_{q=0}^k$. This means that we can choose an atlas on $\mcN$  with local graded homogeneous coordinates $(z_{a_q}^q,\zeta_{b_q}^q)$, where $a_q = 1,\ldots, m_{q}$ and $b_q = 1,\ldots, n_{q}$. Further $z_{a_q}^q$ are even, while $\zeta_{b_q}^q$ are odd coordinates, both of weight $q\in \Zo$.  Transition functions preserve parities and weights of local coordinates. 
Let $\mcR$ be an ideal in $\mcO_{\mcN}$ locally generated  by  $\zeta_{b_q}^{(2q)}$ and $z^{(2q+1)}_{a_q}$, were $q\in \Z$.  

\begin{proposition} 
	The sheaf of ideals $\mcR$ is well-defined. 
\end{proposition}

\begin{proof}
	Consider two graded charts $\mcV$ and $\mcV'$ on $\mcN$ such that  $\mcV_0\cap\mcV_0' \ne \emptyset$ with coordinates $(z_{a_q}^q,\zeta_{b_q}^q)$ and $(\tilde z_{a_q}^q,\tilde\zeta_{b_q}^q)$, respectively.  Let us write the function $\tilde z^{2q+1}_c$ in $\mcV_0\cap\mcV_0'$ in coordinates $(z_{a_q}^q,\zeta_{b_q}^q)$. We have:
	\begin{equation}\label{eq raded transin function}
\tilde z^{2q+1}_c = G_c^{a_1,\ldots, a_u,b_1,\ldots, b_v} z^{(s_1)}_{a_1}\cdots z^{(s_u)}_{a_u} \zeta^{(t_1)}_{b_1}\cdots \zeta^{(t_v)}_{b_v},
	\end{equation}
		where $G^{a_1,\ldots, a_u,b_1,\ldots, b_v}_c$ is a function of weight $0$.
Since $\tilde z^{2q+1}_c$ is an even function of weight $2q+1$, the right hand side of (\ref{eq raded transin function}) is even of weight $2q+1$. 
	 Therefore, 
	$$
	\sum_{i=1}^u s_i+ \sum_{j=1}^v t_j =2q+1.
	$$

Let us prove that the right hand side of (\ref{eq raded transin function}) is in $\mcR$.  If there are odd numbers between numbers $s_i$,  or there are even numbers between $t_j$, we are done.  Let all $s_i$ are even and all $t_j$ are odd. Since their sum is equal to $2q+1$ an odd number, this implies that $v$ is odd. Hence, $G_c^{a_1,\ldots, a_u,b_1,\ldots, b_v}$ is odd as well. This means that $G_c^{a_1,\ldots, a_u,b_1,\ldots, b_v}$ is in $\mcR$, since $\zeta^0_{b_0}\in \mcR$. For the case of $\tilde \zeta_d^{(2q)}$ the argument is similar. 	
\end{proof}

\begin{corollary}\label{cor N super to N graded}
	To any graded supermanifold $\mcN$ we can assign a graded manifold $(\mcN_0,\mcO_{\mcN}/\mcR)$. 
\end{corollary}

\begin{proof}
The ideal $\mcR$ is $\Z$-graded, since it is generated by graded coordinates. We can choose the following graded coordinates in 	$(\mcN_0,\mcO_{\mcN}/\mcR)$:   $(z_{a_s}^s+\mcR,\zeta_{b_t}^t +\mcR)$, where $s$ is even and $t$ is odd. 
\end{proof}

\subsection{The $k$-th-order tangent bundle of a supermanifold $\mcM$ and a $\Z^{\geq 0}$-covering space}

 Let $\mcM$ be a supermanifold. We apply Corollary \ref{cor N super to N graded} to the graded supermanifold $T^{\infty}(\mcM)$.

\begin{theorem}\label{theor main construction 2}
Let $\mcM$ be a supermanifold and $T^{\infty}(\mcM)$ be its 
	$\infty$-th-order tangent bundle. Then 
	$$
	\mcO_{T^{\infty}(\mcM)}/\mcR \simeq \mcO_{\mcP},
	$$
where $\mcP$ is a covering space of $\mcM$. 
\end{theorem}

\begin{proof} The sheaf 
$\mcO_{T^{\infty}(\mcM)}/\mcR$ is a sheaf of a graded manifold. Further denote
\begin{align*}
y^{s}_a:= x^{(s)}_a + \mcR, \quad \eta_b^t: = \xi_b^{(t)} + \mcR, 
\end{align*}  
where $s$ is even and $t$ is odd. Further, we have a natural map, infinite jet prolongation, from $\mcO_{\mcM}$ to $\mcO_{T^{\infty}(\mcM)}$. Then the graded manifold $(\mcM_0, \mcO_{T^{\infty}(\mcM)}/\mcR)$ and the map 
$$
\mcO_{\mcM} \to \mcO_{T^{\infty}(\mcM)} \to \mcO_{T^{\infty}(\mcM)}/\mcR,
$$
which we denote by $\pp$, satisfy Definition \ref{def new-covering of supermanifold}. The proof is complete. 
\end{proof}

\begin{corollary}\footnote{For Lie groups this observation has been conjected by M.~Rotkiewicz in \cite{RVi}.}
	Let $\mcJ_{\mcP}$ be the sheaf of ideals in $\mcO_{\mcP}$ locally generated by odd elements. Then the pure even graded manifolds $(\mcP_0, \mcO_{\mcP}/ \mcJ_{\mcP} )$ and $T^{k}(\mcP_0)$ are naturally isomorphic. 
\end{corollary}

\begin{proof}
	The statement follows from Proposition \ref{prop reduced O_P}. 
\end{proof}

\section{A description of $\mcP_0$, $\mcP_1$ and $\mcP_2$}\label{secP_0, p_1, P_2}

Let $\mcM$ be a supermanifold and $\mcP = \varprojlim \mcP_n$ be a covering of $\mcM$, which always there exists by Theorem \ref{theor construction of a covering for any supermanifld} or by Theorem \ref{theor main construction 2}. Let us describe $\mcP_0$, $\mcP_1$ and $\mcP_2$.

\subsection{A description of $\mcP_0$ and $\mcP_1$}
According to Theorems \ref{theor construction of a covering for any supermanifld} or \ref{theor main construction 2}, the base space of $\mcP$ is equal to the manifold $\mcM_0$, which is the base space of $\mcM$. Further in \cite{Bruce} it was noticed that $T^1\mcM\simeq T\mcM$, the supertangent bundle. Then
$$
\mcO_{\mcP_1} \simeq \mcO_{T\mcM}/(\mcR \cap \mcO_{T\mcM})\simeq \mcO_{\gr\mcM}, 
$$
where $\gr \mcM$ is the retract of $\mcM$, see Section \ref{sec gr}. 
The isomorphism is locally given by 
$$
\dot{\xi}_b + \mcR \cap \mcO_{T\mcM} \longmapsto \xi_b +\mcJ^2.
$$
Let us summerize these results in the following proposition. 

\begin{proposition}\label{prop P_0= P_1=}
	We have  $\mcP_0 \simeq \mcM_0$ and	$\mcP_1 \simeq \gr\mcM.$
\end{proposition}

\subsection{A description of $\mcP_2$}
Let $\mcN$ be any graded manifold of degree $2$. Then we have the following natural projections
$$
\mcN_0 \xleftarrow{\pr^1_0} \mcN_1 \xleftarrow{\pr^2_1} \mcN_2. 
$$
Compare with Section \ref{sec inverse limit}. The graded manifold $\mcN_0$ of degree $0$ is just a usual smooth or complex analytic manifold. Further, $\mcN_1$ is a graded manifold of degree $1$. Any such a graded manifold is isomorphic to the following ringed space (in fact a split supermanifold) 
$$
\big(\mcN_0,\bigwedge (\mcO_{\mcN_1})_1\big),
$$ 
where $(\mcO_{\mcN_1})_1$ is the subsheaf in $\mcO_{\mcN_1}$ of graded functions of weight $1$.

Now assume that $\mcN_0$ and $\mcN_1$ are fixed. Then $\mcN_2$ is determined by the following exact sequence
\begin{equation}\label{eq exact sequence graded man of degree 2}
0 \to  \bigwedge^2 (\mcO_{\mcN_1})_1 \longrightarrow (\mcO_{\mcN_2})_2 \longrightarrow \mcS \to 0. 
\end{equation}
see for instance \cite[Section 4.2, Construction 2, Formula 7]{Vish}. Here $(\mcO_{\mcN_2})_2$ is the subsheaf in $\mcO_{\mcN_2}$ of graded functions of weight $2$ and $\mcS$ is the quotient of $(\mcO_{\mcN_2})_2$ by $\bigwedge^2 (\mcO_{\mcN_1})_1$. Exact Sequence (\ref{eq exact sequence graded man of degree 2})  is an exact sequnce of locally free sheaves. Such a sequnce is determined by its Atiyah class 
$$
\omega_{\mcN_2}\in H^1\big(\mcN_0, \mcS^*\otimes \bigwedge^2 (\mcO_{\mcN_1})_1 \big). 
$$ 
We will call the cohomology class $\omega_{\mcN_2}$ the {\it Atiyah class of the graded manifold $\mcN_2$ of degree $2$}. 

Let $\mcM$ be a supermanifold. Following \cite{Witten Atiyah classes} we denote by 
$$
\omega_2\in H^1(\mcM_0, \bigwedge^2\mathcal E\otimes\mathcal{D}er\mathcal F),
$$ 
where $\gr \mcM \simeq (\mcM_0, \bigwedge \mathcal E)$, 
the first obstruction class to  splitting of $\mcM$, see Appendix \ref{sec First obstruction class omega_2} and Section \ref{sec Donagi and Witten construction}. Let us prove the following theorem. 

\begin{theorem}\label{theor omega_p = omega M}
	Let $\mcM$ be a supermanifold and $\mcP = \varprojlim \mcP_n$ be its $\Z^{\geq 0}$-covering. Then
	$\omega_{\mcP_2} = \omega_2.$
\end{theorem}
\begin{proof}
	First of all let us study Exact Sequence (\ref{eq exact sequence graded man of degree 2}) for $\mcP_2$. By Proposition \ref{prop P_0= P_1=} we have $\mcO_{\mcP_1} \simeq \bigwedge \mcE \simeq \mcO_{\gr \mcM}$. In particular, $\bigwedge^2 (\mcO_{\mcP_1})_1 \simeq \bigwedge^2 \mcE$. Further, by Theorem \ref{theor main construction 2} we have
\begin{align*}
\mcS= (\mcO_{\mcP_2})_2 / (\mcO_{\mcP_1})_2 \simeq
 [(\mcO_{T^2\mcM})_2/(\mcR \cap (\mcO_{T^2\mcM})_2)] / [\mcO_{T\mcM}/(\mcR \cap \mcO_{T\mcM})]_2.
\end{align*}
The sheaf in the right hand side is a locally free sheaf with local basis $\ddot{x}'_a$, where $\ddot{x}'_a$ is the image of standard local coordinate $\ddot{x}_a$ in $T^2\mcM$.  The correspondence $\ddot{x}'_a \mapsto \dd x_a$ determines an isomorphism $\mcS\simeq \Omega$, where $\Omega \simeq (\mathcal{D}er\mathcal F)^*$ is the sheaf of $1$-forms on $\mcM_0$. Now we see that $\omega_{\mcP_2}\in H^1(\mcM_0, \bigwedge^2\mathcal E\otimes\mathcal{D}er\mathcal F)$.

	Let us write the Atiyah cocycle $(h_{\lambda\mu})$ for $\omega_{\mcP_2}$ in local coordinates. 
	Consider two charts $\mathcal U_{\lambda}$ and $\mathcal U_{\mu}$  on $\mcM$ with local coordinates  $(x_a, \xi_b)$ and $(y_a, \eta_b)$. 
	Let in $\mathcal U_{\lambda}\cap \mathcal U_{\mu}$ we have the following transition functions 
	\begin{equation}\label{eq starting smf m=2 4}
	\begin{split}
	&z_c= F_c + \frac12 G_c^{ij} \xi_{i}\xi_{j}+\cdots;\\ 
	&\theta_d = H^{j}_d\xi_j+\cdots,
	\end{split}
	\end{equation}
	where $F_c=F_c(x_a)$, $G_d^{ij}=G_d^{ij}(x_a)$, $H^{j}_i=H^{j}_i(x_a)$ are functions depending only on even coordinates $\{x_a\}$ and we denote by dots the terms in (\ref{eq starting smf m=2 4}) from $\mcJ^3$.

	Let us write the transition functions of $\mcP_2$ in  $\F_2(\mathcal U_{\lambda})\cap \F_2(\mathcal U_{\mu})$. For this we apply $\pp^*$ to both sides of (\ref{eq starting smf m=2 4}) and write transition functions for graded components of weight $0$, $1$ and $2$.  We get
	\begin{equation}\label{eq transit P_2}
	\begin{split}
	\pp^*(z_c)= &\pp^*(F_c + \frac12 G_c^{ij} \xi_{i}\xi_{j} +\cdots ) = \\
	& F_c(y^0_a) + \frac{\partial F_c}{\partial x_l}(y^0_a) y_l^2 + \frac12 G_c^{ij}(y^0_a)  \eta^1_{i}\eta^1_{j} + \cdots; \\
	\pp^*(\theta_d) = &\pp^*(H^{j}_d\xi_j) + \cdots = H^{j}_b(y^0_a)\eta^1_j+ \cdots,
	\end{split}
	\end{equation}
	where $\pp^*(x_a) = y_a^0+y_a^2+\cdots$ and $\pp^*(\xi_b) = \eta_b^1+\eta_b^3+\cdots$ and we denote by $\cdots$ in (\ref{eq transit P_2}) the members of weight $\geq 3$. 
Summing up the transition functions of $\mcP_2$ in  $\F_2(\mathcal U_{\lambda})\cap \F_2(\mathcal U_{\mu})$  are
\begin{equation}\label{eq transit P_2 final}
\begin{split}
\pp^*(z_c)_0&=   F_c(y^0_a); \\
\pp^*(\theta_b)_1 &= H^{j}_b(y^0)\eta^1_j;\\
\pp^*(z_c)_2&=   \frac{\partial F_c}{\partial x_l}(y^0_a) y_l^2 + \frac12 G_c^{ij}(y^0_a)  \eta^1_{i}\eta^1_{j}. \\
\end{split}
\end{equation}
We	denote by $\tilde\psi_{\lambda\mu}$ the automorphism (\ref{eq transit P_2 final}). Denote by $\tilde\phi_{\lambda\mu}$ the following change of coordinates:
\begin{align*}
\pp^*(z_c)_0=   F_c(y^0_a); \quad
\pp^*(\theta_b)_1 = H^{j}_b(y^0)\eta^1_j; \quad \pp^*(z_c)_2=  \frac{\partial F_c}{\partial x_l}(y^0_a) y_l^2. 
\end{align*}
Then $(h_{\lambda\mu} = \tilde\psi_{\lambda\mu} \circ \tilde\phi_{\lambda\mu}^{-1})$ can be regarded as an Atiyah cocycle of $\omega_{\mcP_2}$ written in $\F_2(\mcU_{\lambda})$. 
We see that we have got  the representative $(\omega_2)_{\lambda\mu}$, see (\ref{eq omega_2}). The proof is complete. 
\end{proof}

\begin{corollary}
	Let $\mcM$ be a supermanifold and $\mcP = \varprojlim \mcP_n$ be a its $\Z^{\geq 0}$-covering. The Donagi--Witten construction is equivalent to construction of  the graded manifold  $\mcP_{2}$. 
\end{corollary}
\begin{proof}
	The statement follows from  Theorem \ref{theor omega_p = omega M} and	Theorem \ref{theor Don Witten theorem}. 
\end{proof}

\begin{remark}\label{rem graded define super dim 2}
	Comparing Formulas (\ref{eq starting smf m=2 4}) and (\ref{eq transit P_2 final}), we observe that Formulas (\ref{eq transit P_2 final}) contain the whole information about Formulas (\ref{eq starting smf m=2 4}) modulo  $\mathcal J^3$.
	In other words using Formulas (\ref{eq transit P_2 final}) we can reconstruct Formulas (\ref{eq transit P_2 final}) modulo $\mathcal J^3$, and therefore reconstruct the structure sheaf of the supermanifold $\mcM$ modulo $\mathcal J^3$. 
\end{remark}

\section{About non-split supermanifolds}

Recall that a supermanifold $\mcM=(\mcM_0,\mcO_{\mcM})$ is called {\it split} if $\mcO_{\mcM}\simeq \bigwedge \mcE$, where $\mcE$ is a locally free sheaf on $\mcM_0$. If $\mcO_{\mcM}\not\simeq \bigwedge \mcE$ the supermanifold $\mcM$ is called {\it non-split}. Consider the following example. 

\begin{example}\label{ex superquandric}
	A superquadric $\mathcal Q$ is a supermanifold of dimension $1|2$ with the base space $\mathbb{CP}^1 $. This is a particular case of an isotropic super-Grassmannian introduced in \cite{Manin}. We give a description of $\mathcal Q$ using language of local charts and atlases. For a description of  $\mathcal Q$ as an equation in the protective superspace $\mathbb{CP}^{2|2}$ see \cite[Example 2.14]{OniCOT}. 
	
	 We can cover $\mathbb{CP}^1$ with two standard charts $\mathbb{CP}^1=U_0\cup U_1$.  Let $x$ and $y$ are standard coordinates in $U_0$ and in $U_1$, respectively, with $y=\frac{1}{x}$ in $U_0\cap U_1$.  Consider two superdomains 
	$$
	\mcU_0 = (U_0,\mcF_{U_0}\otimes_{\C} \bigwedge (\xi_1,\xi_2))\quad \text{and} \quad 	\mcU_1 = (U_1,\mcF_{U_1}\otimes_{\C} \bigwedge (\eta_1,\eta_2)),
	$$ 
	where $\mcF_{U_i}$ is the sheaf of holomorphic functions on $U_i$, $i=0,1$, and $\bigwedge (\zeta_1,\zeta_2)$ is the Grassmann algebra over $\C$ generated by $\zeta_1,\zeta_2$.  Summing up, we can regard $(x,\xi_1,\xi_2)$ as coordinates in $\mcU_0$ and $(y,\eta_1,\eta_2)$ as coordinates in $\mcU_1$. Define the supermanifold $\mathcal Q$ in terms of local charts $\mcU_i$ and the following transition functions in $U_0\cap U_1$
	\begin{align*}
	&y= \frac{1}{x} + \frac{1}{x^3}\xi_1 \xi_2;\\
	&\eta_j= \frac{1}{x^2} \xi_j, \,\, j=1,2. 
	\end{align*}
	The obtained supermanifold $\mathcal Q$ is non-split. 
	Historically, this was one of the first examples of such supermanifolds, see \cite{Ber, Green, Manin}.
\end{example}

Non-split supermanifolds plays an important rule in physics. For instance in \cite{Witten not projected} it was shown that for genus $g\geq 5$ the moduli space of super Riemann surfaces is not split. The property of this moduli space to be non-split implies the following \cite[Abstract]{Witten not projected}:

\smallskip

``{\it Physically, this means that certain approaches to superstring perturbation theory that are very powerful in low orders have no close analog in higher orders. Mathematically, it means that the moduli space of super Riemann surfaces cannot be constructed in an elementary way starting with the moduli space of ordinary Riemann surfaces. It has a life of its own.}''
	
	\smallskip
	
{\bf The transition from a supermanifold $\mcM$ to a $\Zo$-covering $\mcP$ resolves in a certain sense the problem of $\mcM$ to be non-split and reduces the study of geometry of $\mcM$ to a classical problem.} 

\smallskip

We will illustrate this in the following example.

	\begin{example}\label{ex covering of Q}
	Denote by $\mcP_{\mathcal Q}$ a $\Zo$-covering of $\mathcal Q$. Then $\mcP_{\mathcal Q}$ is an infinite dimensional graded manifold, with the base space $\mathbb{CP}^1$ and which has the following transition functions 
	\begin{align*}
\text{weight $0$:}\quad	&y= \frac{1}{x};\\
\text{weight $1$:}\quad	&\theta_j= \frac{1}{x^2} \zeta_j, \,\, j=1,2;\\
\text{weight $2$:}\quad	&w= -\frac{1}{x^2}z + \frac{1}{x^3}\zeta_1 \zeta_2;\\
\text{weight $\geq 3$:}\quad	&\cdots
	\end{align*}
	Here the coordinates $x,y$ have weight $0$, $\xi_i, \eta_j$ have weight $1$ and $z$ has weight $2$. (Compare with Proposition \ref{prop P_0= P_1=} and Theorem  \ref{theor omega_p = omega M}.)

	The structure sheaf of $\mcP_{\mathcal Q}$ is determined by a compatible family of vector bundles over $\mathbb{CP}^1$. Let us describe $(\mcP_{\mathcal Q})_0$, $(\mcP_{\mathcal Q})_1$ and $(\mcP_{\mathcal Q})_0$ explicitly, see Proposition \ref{prop P_0= P_1=} and Theorem  \ref{theor omega_p = omega M}. 
	
	 Graded functions of weight $0$ on $\mcP_{\mathcal Q}$ are usual local holomorphic function on $\mathbb{CP}^1$. Further graded functions of weight $1$ are local sections of a vector bundle of rank $2$ isomorphic to $\mcO(-2)\oplus \mcO(-2)$. Graded functions of weight $2$ are sections of a vector bundle $\mcD$, which is determined by the following exact sequence
	$$
	0\to \mcO(-4) \longrightarrow \mcD \longrightarrow \mcO(-2) \to 0
	$$	
with the  Atiyah class 
$$
\omega_{\mcP_{\mathcal Q}}\in H^1(\mathbb{CP}^1,\mcO(-2)^* \otimes \mcO(-4)) \simeq H^1(\mathbb{CP}^1,\mcO(-2))\simeq \C, \quad \omega_{\mcP_{\mathcal Q}} \ne 0.
$$

The second important observation is that the following morphism of sheaves of superalgebras  
$$
(\pp^{(k)})^*: \mcO_{\mathcal Q} \to \mcO_{\mcP^{(k)}_{\mathcal Q}}, \quad k\geq 2,
$$ 
is invective. In other words, the structure sheaf of a non-split supermanifold is realized as a subsheaf of a $\Z$-graded sheaf $\mcO_{\mcP^{(k)}_{\mathcal Q}}$ for any $k\geq 2$. 
\end{example}

The observation from Example \ref{ex covering of Q} holds for any supermanifold. 

\begin{theorem}
	Let $\mcM$ be a (non-split) supermanifold of dimension $m|n$ and $\mcP$ be its $\Zo$-covering. Then the following morphism of sheaves of superalgebras  
	$$
	(\pp^{(k)})^*: \mcO_{\mcM} \to \mcO_{\mcP^{(k)}}, \quad k\geq n,
	$$ 
	is invective.
\end{theorem}
\begin{proof}
	Locally, for superdomains, the morphism $(\pp^{(k)})^*$ is defined by Formula (\ref{eq p^*_k(F) Taylor series}). Therefore $(\pp^{(k)})^*$ is locally invective for any $k\geq n$. But this implies the result, since if the image of a function is $0$, there exists a superdomain where this image is $0$. 
\end{proof}

\section{Supermanifolds of odd dimension $2$}

Recall that the correspondence  $\mcM \mapsto \F_{2}(\mcM)$, see Proposition \ref{prop F_n is a functor},  is a functor  from the category of supermanifolds to the category of graded manifolds of degree $2$. The functor  $\F_2$ restricted to the category of supermanifolds of odd dimension $2$  has  important properties, which we are going to discuss in this section. 

\begin{definition}
	We call a functor $F:C\to D$ an embedding of the category $C$ to the category $D$, if it is faithful and injective on objects. In other words the mapping $F$ is an embedding if it is injective on objects and on morphisms. 
\end{definition}

\begin{theorem}\label{theor we can recover a supermanifold}
	The functor $\F_2$ is an embedding of the category of supermanifolds of odd dimension $2$ into the category of graded manifolds of degree $2$.
\end{theorem}

\begin{proof} 
	By definition, for a supermanifold $\mcM$ the graded manifold $\F_2(\mcM)$ of degree $2$ is equal to $\mcP_2$ from Theorem \ref{theor construction of a covering for any supermanifld}. The correspondence $\mcM \mapsto \mcP_2$  is injective if $\dim \mcM=m|2$. Explicitly the correspondence $\mcM \mapsto \mcP_2$ is given by Formulas  (\ref{eq starting smf m=2 4}) and Formulas (\ref{eq transit P_2 final}). (Note that in Formulas  (\ref{eq starting smf m=2 4}) we do not have $\cdots$ if $\dim \mcM=m|2$.)

	In fact if a graded manifold $\mcN$ of degree $2$ satisfies an additional condition, there is  a supermanifold $\mcM$ of odd dimension $2$ such that $\F_2(\mcM) \simeq \mcN$.  
  Indeed, let $\mcN=(\mcN_0,\mcO_{\mcN})$ be a graded manifold of degree $2$, of graded dimension $(m|2|m)$  such that  in addition the following holds
		\begin{equation}\label{eq O_2/O_1O_1 = Omega}
		(\mcO_{\mcN})_2/((\mcO_{\mcN})_1 \cdot (\mcO_{\mcN})_1) \simeq \Omega^1(\mcN_0),
		\end{equation}
		where $\Omega^1(\mcN_0)$ is the sheaf of $1$-form on $\mcN_0$. Then for any such $\mcN$ there exists unique up to isomorphism supermanifold $\mcM$ of odd dimension $2$ such that $\F_2(\mcM)\simeq \mcN$.
	To any such $\mcN$ we can assign the following exact sequence 
	\begin{equation}\label{eq exact sequence for F(M)}
	0\to (\mcO_{\mcN})_1\cdot (\mcO_{\mcN})_1\to (\mcO_{\mcN})_2 \to \Omega^1(\mcN_0)\to 0.
	\end{equation}
	Consider adapted coordinates $(x_a,z_s= d (x_s))$  in $\T(\mcN_0)$. Taking any local splitting of the short exact sequence \eqref{eq exact sequence for F(M)} we find a local graded chart with coordinates $(x_a, \theta_b, z_s)$, where $(\theta_1,\theta_2)$ is a local basis of the locally free sheaf $(\mcO_{\mcN})_1$. In two such graded charts $\mathcal V$ and $\mathcal V'$ with coordinates $(x_a, \theta_b, z_s)$ and $(x'_{a'}, \theta'_{b'}, z'_{s'})$ we have the following transition function 
	\begin{equation}\label{eq transition function for N of degree 2}
	\begin{split}
	&	x'_{a'}= X_{a'}(x);\\ 
	&\theta'_{b'} =  Y^b_{b'}(x) \theta_b, \,\,\, b'=1,2;\\
	& z'_{s'} = \frac{\partial X_{s'}}{\partial x_b}z_b  +  Z_{s'}\theta_{1}\theta_{2}.
	\end{split}
	\end{equation}
	
	Let $\mathcal U$ and $\mathcal U'$ be two superdomain with $\mathcal U_0=\mathcal V_0$ and $\mathcal U'_0=\mathcal V'_0$ and with coordinates $(x_a,\xi_b)$ and $(x'_{a'},\xi'_{b'})$, respectively. We define in $\mathcal U\cap \mathcal U'$ the following transition functions
	\begin{align*}
	&	x'_{a'}= X_{a'}(x) +  Z_{a'}\xi_{1}\xi_{2} ;\\ 
	& \xi'_{b'} =  Y^b_{b'}(x) \xi_b,\,\,\, b'=1,2.
	\end{align*}
		Let us prove that these transition functions satisfy the cocycle condition. Indeed, consider three charts $\mathcal U_1, \mathcal U_2$ and $\mathcal U_3$ with $\mathcal U_1\cap \mathcal U_2\cap \mathcal U_3\ne \emptyset$. Denote by $T_{ij}$ the transition function  $T_{ij}: \mathcal U_i\to \mathcal U_j$. Consider the following composition of maps
			$$
	T_{31}\circ T_{23}\circ  T_{12} =:  R.
	$$
	Since $\F_2$ is a functor, we get
	$$
	\F_2(T_{31})\circ \F_2(T_{23})\circ \F_2( T_{12}) = \F_2( R).
	$$
	The composition $\F_2(T_{31})\circ \F_2(T_{23})\circ \F_2( T_{12})$ is equal to $\id$, since the cocycle condition for the graded manifold $\mathcal N$ holds true. Therefore $\F_2( R)=\id$. Now we write the morphism $R$ in local coordinates and write the corresponding formulas for $\F_2(R) =\id$ as in Formulas (\ref{eq starting smf m=2 4}) and (\ref{eq transit P_2 final}). We see that if the resulting morphism (\ref{eq transit P_2 final}) is identical, the morphism $R$ is identical as well. A similar argument shows also that $\F_2$ is injective on morphisms. 
\end{proof}

\begin{remark}
	If the odd dimension of $\mcM$ is graded than two then 
	$F_2(\mcM)$ keeps only information on the ringed space $(\mcM_0, \mcO_{\mcM}/\mathcal J^3)$.  
\end{remark}

\begin{remark}
	We saw in the proof of Theorem \ref{theor we can recover a supermanifold} that the functor $\F_2$ is not surjective on objects. We have to add additional conditions (\ref{eq O_2/O_1O_1 = Omega}).  Further, the functor $\F_2$ is not full. 
	In other words, for two supermanifolds $\mcM_1$ and $\mcM_2$, the correspondence 
	$$
	\F_2:\Hom(\mcM_1,\mcM_2)\to \Hom(\F_2(\mcM_1),\F_2(\mcM_2))
	$$
	between hom sets may be not surjective. To see this let us take a split supermanifold $\mcM$ of odd dimension $2$ and any non-trivial bundle automorphism $\gamma$ of $\Omega^1(\mcN_0)$.  By definition of a split supermanifold we can find local coordinates such that transition function (\ref{eq starting smf m=2 4}) can be reduced to the following form
	\begin{align*}
	z_c= F_c ,\quad \theta_d = H^{j}_d\xi_j.
	\end{align*}	
	Then for $\F_2(\mcM)$ in Formulas  (\ref{eq transit P_2 final}) we get that $G_{s'}^{b_1b_2}=0$. Therefore, the $2$-component $(\mcO_{\mcN})_2$ of the structure sheaf of $\mcN:=\F_2(\mcM)$ has the form 
	$$
	(\mcO_{\mcN})_2=  (\mcO_{\mcN})_1\cdot (\mcO_{\mcN})_1 \oplus \Omega^1(\mcN_0).
	$$ 
	In other words, Sequence (\ref{eq exact sequence for F(M)}) is split in this case. Now we can define an automorphism $\Phi $ of $(\mcO_{\mcN})$ determined by 
	$$
	\Phi|_{(\mcO_{\mcN})_0} = \id, \quad \Phi|_{(\mcO_{\mcN})_1} = \id, \quad \Phi|_{(\mcO_{\mcN})_2} =(\id,\gamma). 
	$$
	Clearly such a morphism is not an image of any automorphism of $\mcM$.  
\end{remark}

\section{A $\Z^{\geq 0}$-covering of a Lie superalgebra and a Lie supergroup}

\subsection{A $\Z$-covering  and $\Z^{\geq 0}$-covering of a Lie superalgebra}

We start with the following definitions. If $a\in \Z$, then $\bar a\in \Z_2$ is the parity of $a$.

\begin{definition}\label{def Z covering, algebras}  A $\Z$-covering of a Lie superalgebra $\g$ is a $\Z$-graded superalgebra Lie $\p = \oplus_{a\in \Z} \p_{a}$ together with a homomorphism $p: \p\to \g$ of  Lie superalgebras such that
	$p|_{\p_a}: \p_a \to \g_{\bar a}$ is a bijection for any $a\in \Z$.
\end{definition}

\begin{definition}\label{def Z^0 covering, algebras}  A $\Z^{\geq 0}$-covering of a Lie superalgebra $\g$ is a $\Z$-graded superalgebra Lie $\p = \oplus_{a\geq 0} \p_{a}$ with support $\Z^{\geq 0}=\{ 0,1,2,\ldots\}$ together with a
	homomorphism $p: \p\to \g$ of  Lie superalgebras such that
	$p|_{\p_a}: \p_a \to \g_{\bar a}$ is a bijection for any $a\geq 0$.
\end{definition}

Note that in both cases the Lie bracket of $\p$ is fully determined by the bracket of $\g$. Indeed, for $X\in \p_a$, $X'\in \p_{a'}$ we have $p([X, X']_{\p}) = [p(X), p(X')]_{\g} \in \g_{\bar a +\bar a'}$. Hence
$$
[X, Y]_{\p} = (p|_{\p_{a + a'}})^{-1}([p(X), p(X')]_{\g}).
$$
Moreover,  $p|_{\p_0} : \p_0 \to \g_{\bar 0}$ is a Lie algebra isomorphism and 
$p|_{\p_a}: \p_a\to \g_{\bar a}$ is a morphism of $\g_{\bar 0}$-module, where we identify the Lie superalgebras $\p_0$ and $\g_{\bar 0}$ via $p|_{\p_0}$.

\begin{proposition}\label{p:univ_1} 	
	Let $\g$ be a Lie superalgebra, $\aaa$ be a $\Z$-graded Lie superalgebra (or a $\Z$-graded Lie superalgebra with support $\Z^{\geq 0}$) and let $\psi:\aaa\to \g$ be a homomorphism of Lie superalgebras. If $\p$ is a $\Z$-covering of $\g$ (or $\Z^{\geq 0}$-covering of $\g$), then there exists unique $\Z$-graded homomorphism $\tilde{\psi}: \aaa\to \p$ such that the following diagram is commutative
	$$
	\begin{tikzcd}
	& \p  \arrow[d, "p"]\\
	\aaa \arrow[ur, dashrightarrow, "\exists ! \tilde{\psi}"]   \arrow[r, "\psi"]  & \g
	\end{tikzcd}
	$$
\end{proposition}

\begin{proof}
	Define a linear map $\tilde{\psi}:\aaa\to \p$ such that $\tilde{\psi}(\aaa_s)\subset \p_s$ for $s \in \Z$  by 
	$$
	\tilde{\psi} (X) = (p|_{\p_s})^{-1} (\psi(X)),\quad X\in \aaa_s.
	$$
	 
	We see that  $p\circ \tilde{\psi} = \psi$. Let us check that $\tilde{\psi}$  is a homomorphism. Indeed, let us take $X\in \aaa_s$ and $Y\in \aaa_t$. Then
	\begin{align*}
p ([\tilde{\psi}(X),\tilde{\psi}(Y) ]) =  [p \circ \tilde{\psi}(X),p \circ\tilde{\psi}(Y) ]= [\psi(X),\psi(Y)] =\\
	\psi([X,Y]) = p\circ	\tilde{\psi}([X,Y]).
	\end{align*}
	Since by definition of $\tilde{\psi}$ both $\tilde{\psi}([X,Y])$ and $[\tilde{\psi}(X),\tilde{\psi}(Y)]$ are in $\p_{s+t}$ and $p|_{\p_s}$ is bijective, we get the equality $[\tilde{\psi}(X),\tilde{\psi}(Y) ]= \tilde{\psi}([X,Y])$.
\end{proof}

\begin{proposition}\label{p:lift_of_homomorphism}
	Let $\varphi:\g\to \g'$ be a homomorphism of Lie superalgebras. Then there exists  unique homomorphism $\tilde \varphi$ of  $\Z$-coverings (or $\Z^{\geq 0}$-coverings) $p$ and $p'$  such that the following diagram is commutative
	$$
	\begin{tikzcd}
	\p \arrow[r, "\tilde \varphi"] \arrow["p", d]
	& \p' \arrow[d, "p'" ] \\
	\g\arrow[r,  "\varphi" ]
	& \g'
	\end{tikzcd}
	$$
In particular, any two coverings of a given Lie superalgebra $\g$ are isomorphic.
\end{proposition}
\begin{proof} 
	It follows immediately from Proposition~\ref{p:univ_1}, just take $\psi = \varphi\circ p$. 
\end{proof}

\subsection{Loop superalgebras}\label{sebsec matrix form}
 Let $\g$ be a Lie superalgebra. We can easily see that a $\mathbb Z$-covering $\p'$ and a $\mathbb Z^{\geq 0}$-covering $\p$ of $\g$ is a (sub)algebra of a loop superalgebra in the sense \cite[Definition 3.1.1]{Allison}, see also \cite{Eld}. Indeed, we have
\begin{align*}
\mathfrak p' \simeq \bigoplus_{i\in \Z} \g_{\bar i}\otimes t^i,\quad \mathfrak p'_i \simeq \g_{\bar i}\otimes t^i;\quad \p \simeq \bigoplus_{i\geq 0} \g_{\bar i}\otimes t^i,\quad \p_i \simeq \g_{\bar i}\otimes t^i.
\end{align*}
Here $t$ is a formal even variable.  The covering projection in both cases is given by $p(X\otimes t^i) = X$, where $X\in \g_{\bar i}$. 

Now consider the case $\g = \mathfrak{gl}_{m|n}(\mathbb K)$. The Lie superalgebra $\p$ possesses a simple matrix realization. This implies by Ado's Theorem that we have such a matrix realization for any finite dimensional Lie superalgebra $\mathfrak{h}\subset \mathfrak{gl}_{m|n}(\mathbb K)$.  Recall that the general linear Lie superalgebra $\mathfrak{gl}_{m|n}(\mathbb K)$ contains all matrices over $\K$ in the following form
$$
\left(
\begin{array}{cc}
A& B \\
C& D\\
\end{array}\right),
$$
where $A$ is $m\times m$-matrix and  $D$ is $n\times n$-matrix over $\mathbb K$.  Then the Lie superalgebra $\mathfrak p$ contains all matrices in the following form
$$
\left(
\begin{array}{ccccc}
A_0 & 0 & 0& 0& \cdots\\
C_1& D_0& 0 & 0&\cdots\\
A_2&B_1& A_0 & 0& \cdots\\
C_3 &D_2& C_1 & D_0& \cdots\\
\cdots& \cdots&  \cdots &  \cdots& \cdots\\
\end{array}\right).
$$
Here $A_i$ are $m\times m$-matrices, $D_i$ are $n\times n$-matrices, $B_i$ are $m\times n$-matrices, $C_i$ are $n\times m$-matrices over $\mathbb K$.
We have
$$
\p = \bigoplus_{n\geq 0} \p_n,\quad  \p=  \p_{\bar 0} \oplus  \pp_{\bar 1},
$$
where $\p_{\bar 0}$ contains all matrices with $B_i=0$ and $C_j=0$ for any $i,j$, while $\p_{\bar 1}$ contains all matrices with $A_i=0$ and $D_j=0$ for any $i,j$. The $\mathbb Z$-grading of $\mathfrak p$ has natural meaning: if $X\in A_i$, then the $\Z$-grading of $X$ is $i$. The same for other blokes $B_i$, $C_j$ and $D_k$. The covering map $p$ is given by $(A_i,D_i)\mapsto  (A,D)$ and $(B_j,C_j)\mapsto  (B,C)$.

We can explicitly construct the corresponding Lie supergroup.  
Let $\mcG$ be a Lie subsupergroup in the general linear Lie supergroup $\mathrm {GL}_{m|n}(\mathbb K)$.  Let $\mathrm {GL}_{m|n}(\mathbb K)$ has the following coordinate superdomain $\mathcal U$
$$
\left(
\begin{array}{cc}
X & \Xi\\
\mathrm H& Y\\
\end{array}\right),
$$
where $X\in \mathrm {GL}_{m}(\mathbb K)$, $Y\in \mathrm {GL}_{n}(\mathbb K)$ are matrices of even coordinates of $\mathcal U$ and $\Xi$, $\mathrm H$ are matrices of odd coordinates. Now the Lie supergroup of $\p$ has the following coordinate superdomain
$$
\left(
\begin{array}{ccccc}
X_0 & 0 & 0& 0& \cdots\\
\mathrm H_1& Y_0& 0 & 0&\cdots\\
X_2&\Xi_1& X_1 & 0& \cdots\\
\Xi_3&Y_2& \mathrm H_1 & Y_0& \cdots\\
\cdots& \cdots&  \cdots &  \cdots& \cdots\\
\end{array}\right).
$$
$X_i\in \mathrm {GL}_{m}(\mathbb K)$, $Y_j\in \mathrm {GL}_{n}(\mathbb K)$ are matrices of even coordinates  and $\Xi_s$, $\mathrm H_t$ are matrices of odd coordinates. The Lie supergroup multiplication is given by usual matrix multiplication.

\section{A $\mathbb Z^{\geq 0}$-covering for a  Lie supergroup}\label{sec Z-covering, Lie superalgebra and Lie supergroup}

\subsection{A $\mathbb Z^{\geq 0}$-covering for a  Lie supergroup}\label{sec Z-covering for a  Lie supergroup} 
Let $\mcG$ be a Lie supergroup. In particular $\mcG$ is a supermanifold and  we can find its $\mathbb Z^{\geq 0}$-covering $\mcP$. Note that now $\mcP$ does not have any graded Lie group structure. We can define a structure of a Lie supergroup on $\mcP$ using the supergroup morphisms $\mu$, $\kappa$, $\varepsilon$ in $\mcG$.  First of all a $\Z^{\geq 0}$-covering of $\mcG\times \mcG$ is isomorphic to $\mcP\times \mcP$, which is defined as an inverse limit of $\mcP_n\times \mcP_n$. Indeed, $\mcP\times \mcP$ satisfies the universal properties of a $\Z^{\geq 0}$-covering of $\mcG\times \mcG$.

Now we use Theorem \ref{theor morphism psi can be covered by Psi}.  Denote by $\tilde\mu:\mcP\times \mcP\to \mcP$, $\tilde\kappa: \mcP\to \mcP$, $\tilde\varepsilon:\K\to \mcP$ the lifts of morphisms $\mu$, $\kappa$, $\varepsilon$, respectively. (Clearly a $\Z^{\geq 0}$-covering of $\K$ is just $\K$.) An standard argument shows that $\tilde\mu$, $\tilde\kappa$, $\tilde\varepsilon$ satisfy the supergroup axioms. For example let us prove associativity axiom. The following two morphisms 
$$
\tilde\mu\circ (\tilde\mu\times \id)\quad \text{and} \quad \tilde\mu\circ (\id \times \tilde\mu)
$$
are equal, since they both cover the morphism $\mu\circ (\mu\times \id)= \mu\circ (\id \times \mu)$. Summing up, we showed that $\mcP$ is a graded Lie group. Since $\tilde\mu$ is a cover of $\mu$, by Theorem \ref{theor morphism psi can be covered by Psi}  we have $\tilde\mu \circ (\pp\times \pp) = \pp\circ \mu$. Therefore, $\pp$ is a homomorphism of Lie supergroups.

Let $\mathcal A$ be a graded group of degree $n$ and $\psi:\mathcal A\to \mcG$ be a homomorphism of Lie supergroups.  By Theorem \ref{theor univ_property supermanifold} there exists $\Psi$ such that $\psi = \Psi\circ \pp$.  Let us prove that $\Psi$ is a homomorphism of graded Lie groups. Consider two morphisms $$
\tilde\mu \circ  (\Psi\times \Psi): \mathcal A\times\mathcal A \to \mcP \quad  \text{and} \quad  \Psi \circ \mu_{\mathcal A} : \mathcal A\times\mathcal A \to \mcP,
$$ 
where $\mu_{\mathcal A}$ is the multiplication in $\mathcal A$. They are lifts of the morphisms $\mu \circ  (\psi\times \psi)$ and $\psi \circ \mu_{\mathcal A} $, respectively. Since $\psi$ is a  homomorphism, in other words $\mu \circ  (\psi\times \psi)=\psi \circ \mu_{\mathcal A} $ we obtain the result by Theorem \ref{theor univ_property supermanifold}. Now let $\g$, $\p$ and $\aaa$ be Lie superalgebras of $\mcG$, $\mcP$ and $\mathcal A$. We see that $\p$ satisfies Proposition \ref{p:univ_1}, therefore, $\p \simeq \bigoplus_{i\geq 0} \g_{\bar i}\otimes t^i$. 

\begin{theorem}
	Let $\mcG$ be a Lie supergroup and $\mcP$ be a $\Zo$-graded covering of $\mcG$. Then we have
	$$
	\mcP \simeq (\mcG_0,\mathrm{Hom}_{\mathcal U(\p_{0})} (\mathcal U(\p), \mathcal F_{\mcG_0})).
	$$
	In other words $\mcP$ is determined by the Harish-Chandra pair $(\mcG_0,\p)$. 
\end{theorem}

\begin{proof}
	The idea is to show that the graded Lie group  $ (\mcG_0,\mathrm{Hom}_{\mathcal U(\p_{0})} (\mathcal U(\p), \mathcal F_{\mcG_0}))$ together with the Lie supergroup morphism $ \pp': (\mcG_0,\mathrm{Hom}_{\mathcal U(\p_{0})} (\mathcal U(\p), \mathcal F_{\mcG_0})) \to \mcG$ given by
	\begin{align*}
	\pp'^*: \mathrm{Hom}_{\mathcal U(\g_{\bar 0})} (\mathcal U(\mathfrak \g), \mathcal F_{\mcG_0}) \simeq \mcO_{\mcG} & \longrightarrow\mathrm{Hom}_{\mathcal U(\p_{0})} (\mathcal U(\mathfrak \p), \mathcal F_{\mcG_0}) ;\\
	\pp'^* (f)(X)(g)& = f(p (X)) (g),\quad X\in \mathcal U(\mathfrak \p), \quad g\in \mcG_0,
	\end{align*}
	where $p$ is as in Proposition \ref{p:univ_1},
	 satisfies the universal property and to use Theorem \ref{theor def of cov using univ prop}. Let us take a non-negatively graded Lie group $\mcA$ of degree $n$. By \cite{Kotov} this graded Lie group is determined by its graded Harish-Chandra pair. In particular the structure sheaf $\mcO_{\mcA}$ of $\mcA$ is isomorphic to 
	$$
	\mathrm{Hom}_{\mathcal U(\aaa_{0})} (\mathcal U(\mathfrak \aaa), \mathcal F_{\mcA_0}),
	$$
	where $\mcF_{\mcA_0}$ is the structure sheaf of the Lie group $\mcA_0$ and $\aaa=\Lie \mcA$. 
	
	Further, similarly to \cite{Kotov,ViLieSupergroup} we can show that any Lie supergroup isomorphism $\psi: \mcA\to \mcG$ is determined by the following formula
	\begin{align*}
	\psi^*: \mathrm{Hom}_{\mathcal U(\g_{\bar 0})} (\mathcal U(\mathfrak \g), \mathcal F_{\mcG_0}) \simeq \mcO_{\mcG} & \longrightarrow\mathrm{Hom}_{\mathcal U(\aaa_{0})} (\mathcal U(\mathfrak \aaa), \mathcal F_{\mcA_0}) ;\\
	\psi^* (f)(X)(a)& = f(\phi (X)) (\psi_0(a)),\quad X\in \mathcal U(\mathfrak \aaa), \quad a\in \mcA_0,
	\end{align*}
	where $\phi: \mathcal U(\mathfrak \aaa)\to \mathcal U(\mathfrak \g) $ is the morphism of universal enveloping algebras corresponding to the Lie supergroup morphism $\psi$. Let us define the lift $\Psi: \mcA \to (\mcG_0,\mathrm{Hom}_{\mathcal U(\p_{0})} (\mathcal U(\p), \mathcal F_{\mcG_0}))$ of $\psi$. By Proposition \ref{p:univ_1} there exists a lift $\tilde \phi: \aaa\to \p$ of $\phi$. We put
 \begin{align*}
 \Psi^*: \mathrm{Hom}_{\mathcal U(\p_{0})} (\mathcal U(\p), \mathcal F_{\mcG_0}) & \longrightarrow\mathrm{Hom}_{\mathcal U(\aaa_{0})} (\mathcal U(\mathfrak \aaa), \mathcal F_{\mcA_0}) ;\\
 \Psi^* (f)(X)(a)& = f(\tilde\phi (X))(\psi_0(a)),\quad X\in \mathcal U(\mathfrak \aaa) , \quad a\in \mcA_0.
 \end{align*}
 A standard argument, see for example \cite{ViLieSupergroup}, shows that the formula above defines a homomorphism of Lie supergroups.  Further we have $\psi^* = \Psi^* \circ \pp'^*$, since $\phi= p \circ \tilde \phi$. The prove is complete.
\end{proof}

\appendix

\section{First obstruction class $\omega_2$}\label{sec First obstruction class omega_2}

Let us describe the first obstruction class to splitting a supermanifold using results \cite{Ber,Green,Oni,Roth}. We follow the exposition of \cite{Oni}. First of all consider a split supermanifold  $\mcM=(\mcM_0,\mcO)$, where $\mcO\simeq\bigwedge \mathcal E$ and $\mathcal E$ is a locally free sheaf on $\mcM_0$. Denote by $\mathcal{D}er\mathcal O$ the sheaf of vector fields on $\mcM$ and by $\mathcal{D}er\mathcal F$ the sheaf of vector fields on the underlying space $\mcM_0$.  The sheaf $\mathcal{D}er\mathcal O=\bigoplus\limits_{p\geq -1} \mathcal{D}er_p\mathcal O$ is naturally $\Z$-graded since $\bigwedge \mathcal E$ is $\Z$-graded.  We have the following exact sequence
\begin{equation}\label{eq exact seq}
0\to \bigwedge^3\mathcal E\otimes\mathcal E^* \xrightarrow{\imath} \mathcal{D}er_2\mathcal O \xrightarrow{\jmath} \bigwedge^2\mathcal E\otimes\mathcal{D}er\mathcal F \to 0,
\end{equation}
see \cite[Formula (5)]{Oni}. The map $\imath$ extends any linear map  $\mathcal E\to  \bigwedge^3\mathcal E$ to a vector field in $\mathcal{D}er_2\mathcal O$  via Leibniz rule. The map $\jmath$ is the restriction  $v \mapsto v|_{\mcF}$ of a vector field $v$ to the structure sheaf $\mcF$ of $\mcM_0$.

Let $\mathcal M'= (\mcM_0,\mcO_{\mcM'})$ be a (non-split) supermanifold with the base space $\mcM_0$. 
Thanks to Green \cite{Green}   all supermanifolds $\mcM'$ corresponding to a fixed locally free sheaf $\mcE$ over $\mcM_0$, that is $\gr (\mcM')\simeq (\mcM_0, \mcO)$,  can be classified as orbits of certain action which we are going to recall. Denote by $\mathcal{A}ut \mathcal O$ the sheaf of automorphisms of $\mcO$. Note that by definition any automorphism is even and maps a stalk $\mcO_x$, $x\in \mcM_0$, to itself. 
Consider the  subsheaf $\mathcal{A}ut_{(2)} \mathcal O$  of $\mathcal{A}ut \mathcal O$ defined as
$$
\mathcal{A}ut_{(2)} \mathcal O = \{ a\in \mathcal{A}ut \mathcal O\,\,|\,\, a(u)-u\in \mathcal J^2\,\, \text{for any}\,\, u\in \mathcal O \},
$$
see also \cite[Formula (17)]{Oni}. Recall that $\mathcal J$ is the sheaf of ideals generated by odd elements in $\mathcal O$.  Denote by $\mathrm{Aut}\mathcal E$ the group of global automorthisms of $ \mcE$.  There is  a natural action of $\mathrm{Aut}\mathcal E$ on the sheaf $\mathcal{A}ut_{(2)} \mathcal O$ given by $a\mapsto \phi\circ a\circ \phi^{-1} $ for $\phi \in \mathrm{Aut}\mathcal E$, see \cite[Section 1.4]{Oni}. This action induces an action of $\mathrm{Aut}\mathcal E$ on the set $H^1(\mcM_0, \mathcal{A}ut_{(2)} \mathcal O)$. 
By Green \cite{Green} 
the set of orbits $H^1(\mcM_0, \mathcal{A}ut_{(2)} \mathcal O)/ \mathrm{Aut}\mathcal E$ are in one-to-one correspondence with  isomorphism classes of supermanifolds $\mcM'$ such that $\gr (\mcM')\simeq (\mcM_0, \bigwedge \mathcal E)$. More precisely to any supermanifold $\mcM'$ such that $\gr (\mcM')\simeq (\mcM_0, \bigwedge \mathcal E)$ we can assign  a class $\gamma\in H^1(\mcM_0, \mathcal{A}ut_{(2)} \mathcal O)$. If $\gamma_i$ is the class corresponding to a supermanifold $\mcM'_i$ such that $\gr (\mcM'_i)\simeq \mcM$, where $i=1,2$, then $\mcM'_1\simeq \mcM'_2 $ if and only if $\gamma_1$ and $\gamma_2$ are in the same orbit of $\mathrm{Aut}\mathcal E$.

In \cite{Roth} the following map of sheaves was defined
\begin{equation}\label{eq Roth map}
\lambda_2:\mathcal{A}ut_{(2)} \mathcal O \to \mathcal{D}er_2\mathcal O,
\end{equation}
where $\lambda_2(a)$ is the $2$-component of $\log a\in \mathcal{D}er\mathcal O$,
see also \cite[Formula (19)]{Oni}. Combining the map (\ref{eq Roth map}) and the map $\jmath:\mathcal{D}er_2\mathcal O \to \bigwedge^2\mathcal E\otimes\mathcal{D}er\mathcal F$ from (\ref{eq exact seq}) we get the following map of cohomology sets
$$
H^1(\mcM_0, \mathcal{A}ut_{(2)} \mathcal O) \to H^1(\mcM_0, \bigwedge^2\mathcal E\otimes\mathcal{D}er\mathcal F)
$$
and the corresponding map of $\mathrm{Aut}\mathcal E$-orbits 
$$
H^1(\mcM_0, \mathcal{A}ut_{(2)} \mathcal O)/\mathrm{Aut}\mathcal E \to H^1(\mcM_0, \bigwedge^2\mathcal E\otimes\mathcal{D}er\mathcal F)/\mathrm{Aut}\mathcal E.
$$
If a class $\gamma\in H^1(\mcM_0, \mathcal{A}ut_{(2)} \mathcal O)$ corresponds to a non-split supermanifold $\mcM'$, the image of $\gamma$ in $H^1(\mcM_0, \bigwedge^2\mathcal E\otimes\mathcal{D}er\mathcal F)$ is called {\it the first obstruction class} to  splitting of $\mcM'$ and following \cite[Section 2.1]{Witten Atiyah classes} we denote this class by $\omega_2$. 

\begin{example}
	Consider the case when $\mcM'$ has odd dimension $2$ in details. Since $\bigwedge^3 \mathcal E=\{0\}$, from (\ref{eq exact seq}) it follows that 
	$$
	\mathcal{D}er_2\mathcal O\simeq \bigwedge^2\mathcal E\otimes\mathcal{D}er\mathcal F.  
	$$
	In this case the map (\ref{eq Roth map}) is an isomorphism. Therefore we have the following set bijection
	$$
	H^1(\mcM_0, \mathcal{A}ut_{(2)} \mathcal O) \simeq H^1(\mcM_0, \bigwedge^2\mathcal E\otimes\mathcal{D}er\mathcal F)
	$$
	and the corresponding bijection of the sets of $\mathrm{Aut}\mathcal E$-orbits. 
	Now Green's result \cite{Green} implies that $\omega_2$ is the only obstruction for a supermanifold to be split in this case. In other words a supermanifold $\mcM'$ of odd dimension $2$ such that $\gr(\mcM')\simeq (\mcM_0, \bigwedge \mathcal E)$ is split if and only if $\omega_2=0$. Note that the notion of a split and a projectable, see \cite{Witten not projected} for details, supermanifold coincide in this case. 
\end{example}

Let us compute a representative cocycle of $\omega_2$. First of all let us describe Green cocycle $(g_{ij})$ of $\mcM'$. We follow \cite[Section 4]{OniCOT}. Consider two charts $\mathcal U_i$ and $\mathcal U_j$  on $\mcM'$  with local coordinates  $(x_a, \xi_b)$ and $(y_a, \eta_b)$. 
Let in $\mathcal U_i\cap \mathcal U_j$ we have the following transition functions 
\begin{equation}\label{eq starting smf m=2 3}
\begin{split}
&y_a= F_a + \frac12 G_a^{ij} \xi_{i}\xi_{j}+\cdots,\,\,\, a=1,\ldots, p;\\ 
&\eta_b = H^{j}_b\xi_j+\cdots, \,\,\, b=1,\ldots, q,
\end{split}
\end{equation}
where $F_a=F_a(x)$, $G_a^{ij}=G_a^{ij}(x)$, $H^{j}_i=H^{j}_i(x)$ are functions depending only on even coordinates $\{x_a\}$ and we denote by dots terms in (\ref{eq starting smf m=2 3}) from $\mcJ^3$. 

Denote by $(x'_a, \xi'_b)$ and $(y'_a, \eta'_b)$ the images of coordinates $(x_a, \xi_b)$ and $(y_a, \eta_b)$ in $\gr\mcO$. Then the transition functions in $\gr\mathcal U_i\cap \gr\mathcal U_j$ are given by the following formulas
\begin{equation}\label{eq starting smf gr}
\begin{split}
&y'_a= F_a(x') \,\,\, a=1,\ldots, p;\\ 
&\eta'_b = H^{j}_b(x')\xi'_j, \,\,\, b=1,\ldots, q,
\end{split}
\end{equation}
Denote the automorphism (\ref{eq starting smf m=2 3}) by
 $\psi_{ij}$ and the automorphism (\ref{eq starting smf gr}) 
  by $\phi_{ij}$. In \cite[Section 4]{OniCOT} the following formula was obtained
  $$
  \psi_{ij} = g_{ij} \circ \phi_{ij}\quad \text{or}\quad g_{ij} = \psi_{ij}\circ   \phi_{ij}^{-1},
 $$
where $g_{ij}$ is the Green cocycle written in coordinates $(x_a, \xi_b)$. To obtain a representative cocycle $((\omega_2)_{ij})$ of the cohomology class $\omega_2$ we apply $\lambda_2$ to $g_{ij}$. Explicitly we have
 \begin{equation}\label{eq omega_2}
(\omega_2)_{ij} = \jmath\circ\lambda_2(g_{ij}),
\end{equation}
where $\jmath$ is as in (\ref{eq exact seq}). 

	\section{A Donagi--Witten construction}\label{sec Donagi and Witten construction}

In \cite[Section 2.1]{Witten Atiyah classes} Donagi and Witten gave a description of the first obstruction class $\omega=\omega_2$ to splitting a supermanifold via differential operators. In this section we remind the Donagi--Witten construction. (For a definition of $\omega_2$ see Appendix \ref{sec First obstruction class omega_2}.) Further we give an interpretation of the Donagi--Witten construction using the language of  double vector bundles and graded manifolds of degree $2$. At the end using differential forms instead of differential operators we dualize this construction, which allows us to find its higher analogue.

Let $\mcM$ be a (non-split) supermanifold. 
In \cite[Section 2.1]{Witten Atiyah classes} the first obstruction class $\omega_2$ to splitting a supermanifold $\mcM=(\mcM_0, \mathcal O)$ was interpreted via a certain sheaf  of differential operators on $\mcM$. Let $\mcE=\mcO/\mcJ$ be the corresponding to $\mcM$ locally free sheaf. In Section \ref{sec First obstruction class omega_2} we defined $\omega_2$ as an element in 
$$
\omega_2\in H^1(\mcM_0, \bigwedge^2\mathcal E\otimes\mathcal{D}er\mathcal F).
$$ 
Any such class $\omega_2$ can be interpreted as the Atiyah class (extension class)  of the follo\-wing short exact sequence of locally free sheaves
\begin{equation}\label{eq DW exact sequence}
0\to \mathcal{D}er\mathcal F\to \mathcal D_{\omega_2} \to \bigwedge^2\mathcal E^* \to  0. 
\end{equation}
compare with  \cite[Formula 2.1]{Witten Atiyah classes}. In other words, the locally free sheaf $\mathcal D_{\omega_2}$ is defined by the class $\omega_2$. Further, in \cite[Section 2.1]{Witten Atiyah classes} it was noticed that the locally free sheaf $\mathcal D_{\omega_2}$ can be realized via a sheaf of differential operators, which we are going to recall. The construction of the exact sequence (\ref{eq DW exact sequence}) with the obstruction class $\omega_2$ using differential operators we call the {\it Donagi-Witten construction}. This construction was suggested in \cite[Section 2.1]{Witten Atiyah classes}.

To get  $\mathcal D_{\omega_2}$ we consider the locally free sheaf of second order differential operators on $\mcM$ with values in $\mcO/\mcJ$ whose symbol are totally odd. We remind  this construction using charts and local coordinates.  Consider two charts $\mathcal U_1$ and $\mathcal U_2$  on $\mcM$ with non-empty intersection and with local coordinates  $(x_a, \xi_b)$ and $(y_a, \eta_b)$. 
Let in $\mathcal U_1\cap \mathcal U_2$ we have the following transition functions 
\begin{equation}\label{eq starting smf m=2}
\begin{split}
&y_a= F_a + \frac12 G_a^{ij} \xi_{i}\xi_{j}+\cdots\,\,\, a=1,\ldots, p;\\ 
&\eta_b = H^{j}_b\xi_j+\cdots, \,\,\, b=1,\ldots, q,
\end{split}
\end{equation}
where $F_a=F_a(x)$, $G_a^{ij}=G_a^{ij}(x)$, $H^{j}_i=H^{j}_i(x)$ are 
functions depending only on even coordinates $\{x_a\}$ and we denote by dots terms in (\ref{eq starting smf m=2}) from $\mcJ^3$.  (Here we use Einstein summation notation. For example, $H^{j}_b\xi_j$ means the sum $\sum\limits_{j=1}^q H^{j}_b\xi_j$.)  If $R: \mcO \to \mcO$ is a differential operator we denote by $\rd{R}: \mcO \to \mcF$ the composition of $R$ with the map $\red:\mcO \to \mcO/\mcJ  =\mcF$.

Following Donagi and Witten, see  \cite[Section 2.1]{Witten Atiyah classes}, we define the locally free sheaf $\mathcal D_{\omega_2}$ on $\mcM_0$ as the sheaf, which is locally generated over $\mathcal F$  by 
$$
\langle \rd{\partial_{x_a}},  \rd{\partial_{\xi_i}\circ {\partial_{\xi_j}}} \rangle.
$$ 
In \cite[Theorem 2.5]{Witten Atiyah classes} it was shown that this definition does not depend on local coordinates. Note that the operators 
$\partial_{\xi_i}$ anticommute, i.e. $\partial_{\xi_i}\circ {\partial_{\xi_j}} = - \partial_{\xi_j}\circ {\partial_{\xi_i}}$.
Let us write transition function for  generators $\rd{\partial_{x_a}}$,  $\rd{\partial_{\xi_i}\circ {\partial_{\xi_j}}}$ of $\mathcal D_{\omega_2}$. 
We have
\begin{align*}
\partial_{x_b} &= \frac{\partial y_a}{\partial x_b} \partial_{y_a} + \frac{\partial \eta_i}{\partial x_b} \partial_{\eta_i} = \frac{\pa F_a}{\pa x_b}  \pa_{y_a} + \frac{\pa H_i^j}{\pa x_b}  \xi_j \pa_{\eta_i} + \ldots; \\
\pa_{\xi_b} &= \frac{\partial y_a}{\partial \xi_b} \pa_{y_a}  +  \frac{\partial \eta_i}{\partial \xi_b} \pa_{\eta_i} = G_a^{bi} \xi_i \pa_{y_a} +  H_i^b \pa_{\eta_i} + \ldots.
\end{align*}
Note that the functions $\frac{\pa F_a}{\pa x_b} =\frac{\pa F_a}{\pa x_b} (x(y)) $, $\frac{\pa H_i^j}{\pa x_b} =\frac{\pa H_i^j}{\pa x_b} (x(y))$, $G_a^{bi}= G_a^{bi}(x(y))$ and $H_i^b =H_i^b(x(y))$ are written in coordinates $\{y_b\}$ and  $x(y) $ is the inverse of the map $y_a= F_a(x)$. We denoted here by dots the terms that contain cooeficients from $\mcJ^2$. 

Further calculating $\pa_{\xi_{j_1}} \circ \pa_{\xi_{j_2}}$ and factirizing all term with coefficients from  $\mcJ$ we find 
the transition functions for $\mathcal D_{\omega_2}$ explicitly: 
\begin{equation}\label{eq Don Witten}
\begin{split}
\rd{\pa_{x_b}} &= \frac{\pa F_a}{\pa x_b}\rd{\pa_{y_a}} \\
\rd{\pa_{\xi_{b_1}} \circ \pa_{\xi_{b_2}}}  &= G_a^{b_2b_1}  \rd{\pa_{y_a}}  + H_{i_1}^{b_1} H_{i_2}^{b_2}\,\, \rd{\pa_{\eta_{i_1}} \circ \pa_{\eta_{i_2}}} .
\end{split}
\end{equation}
Here we used the equation $\rd{\pa_{\xi_b}} =   H_i^b \rd{\pa_{\eta_i}}$.
We see that the sheaf $\mathcal{D}er\mathcal F$ is a subsheaf of $\mathcal D_{\omega_2}$ and that the factor $\mathcal D_{\omega_2}/\mathcal{D}er\mathcal F$ is isomorphic to $\bigwedge^2\mathcal E$, compare with (\ref{eq DW exact sequence}). 

\begin{theorem}\cite[Section 2.1]{Witten Atiyah classes}\label{theor Don Witten theorem}
	The Atiyah class of the vector bundle $\mathcal D_{\omega_2}$ is equal to $\omega_2$. 
\end{theorem}

\bigskip

\noindent
E.~V.: Departamento de Matem{\'a}tica, Instituto de Ci{\^e}ncias Exatas,
Universidade Federal de Minas Gerais,
Av. Ant{\^o}nio Carlos, 6627, CEP: 31270-901, Belo Horizonte,
Minas Gerais, BRAZIL, Institute of Mathematics, University of Cologne, Weyertal 86-90, 50931 Cologne, GERMANY, and Laboratory of Theoretical and Mathematical Physics, Tomsk State University,
Tomsk 634050, RUSSIA.

\noindent Email: {\tt VishnyakovaE\symbol{64}googlemail.com}

\end{document}